\documentclass[12pt]{amsart}%

\usepackage{amssymb}
\usepackage{amsmath}
\usepackage{amsthm}
\usepackage{amscd,mdwlist}
\usepackage[margin=1in]{geometry}%
\usepackage{hyperref}%[hypertex]
\usepackage{graphicx} 
\usepackage{enumitem,graphics, graphicx,tikz, pst-node, pst-plot, pstricks}
\usetikzlibrary{calc,patterns,shapes,decorations.pathreplacing}
\usetikzlibrary{decorations.fractals}
\usetikzlibrary{decorations.footprints}

\theoremstyle{plain}
\newtheorem{theorem}{Theorem}[section]
\newtheorem{proposition}{Proposition}[section]
\newtheorem{corollary}{Corollary}[section]
\newtheorem{lemma}{Lemma}[section]
\newtheorem{definition}{Definition}[section]

\theoremstyle{remark}
\newtheorem{remark}{Remark}[section]

\newtheorem{assumption}{Assumption}[section]

\DeclareMathOperator{\diverg}{div}

\DeclareMathOperator{\im}{Im}

\DeclareMathOperator{\Tr}{Tr}
\DeclareMathOperator{\Ext}{Ext}

\newcommand{\A}{\ensuremath{\mathcal A}}

\newcommand{\E}{\ensuremath{\mathcal E}}
\newcommand{\T}{\ensuremath{\mathcal T}}
\newcommand{\oO}{\ensuremath{ \overline
  {\Omega}}}
\newcommand{\pO}{\ensuremath{ \partial {\Omega}}}
\newcommand{\EpO}{\ensuremath{ \E_{\partial\Omega}}}

\begin{document}

\title{Fractal snowflake domain diffusion with boundary and interior drifts}

\author{Michael Hinz$^1$}
\address{$^1$Fakult\"at f\"ur Mathematik, Universit\"at Bielefeld, Postfach 100131, 33501 Bielefeld, Germany}
\email{mhinz@math.uni-bielefeld.de}
\thanks{$^1$Research supported in part by SFB 701 of the German Research Council (DFG)}

\author{Maria Rosaria Lancia$^2$}
\address{$^2$Dipartimento di Scienze di Base e Applicate per l'Ingegneria,
Sapienza, Universita di Roma
Via A. Scarpa 16, 00161 Roma, Italy}
\email{maria.lancia@sbai.uniroma1.it}
\thanks{$^2$, $^4$ The authors have been supported by the Gruppo
Nazionale per l'Analisi Matematica, la Probabilit\`a e le loro
Applicazioni (GNAMPA) of the Istituto Nazionale di Alta Matematica
(INdAM)}

\author{Alexander Teplyaev$^3$}
\address{$^3$Department of Mathematics, University of Connecticut, Storrs, CT 06269-3009 USA}
\email{teplyaev@uconn.edu}
%\thanks{$^3$Research supported in part by }

\author{Paola Vernole$^4$}
\address{$^4$Dipartimento di Matematica,
Universita degli Studi di Roma âLa Sapienzaâ,
P.zale Aldo Moro 2, 00185 Roma, Italy}
\email{vernole@mat.uniroma1.it}

\date{\today}

\begin{abstract}
We study a parabolic Ventsell problem for a second order differential operator in divergence form and with interior and boundary drift terms on the snowflake domain. 
We prove that under standard conditions a related Cauchy problem possesses a unique classical solution and explain in which sense it solves a rigorous formulation of the initial Ventsell problem. As a second result we prove that functions that are intrinsically Lipschitz on the snowflake boundary admit Euclidean Lipschitz extensions to the closure of the entire domain. Our methods combine the fractal membrane analysis, the vector analysis for local Dirichlet forms and PDE on fractals, coercive closed forms, and the analysis of Lipschitz functions.

\tableofcontents\end{abstract}\maketitle

\section{Introduction}

The main objective of this paper is to study a parabolic Ventsell problem for a second order differential operator in divergence form with measurable coefficients and drift terms in the interior and with drift and diffusion terms on the boundary of the classical snowflake domain $\Omega$, see Figure~\ref{fig-snow}. This continues the investigation of diffusion problems related to fractal membranes started in \cite{La02} (see \cite{CLa14,LaV14} for the most recent results and references).

  \begin{figure}
 \includegraphics %[angle=90]
 {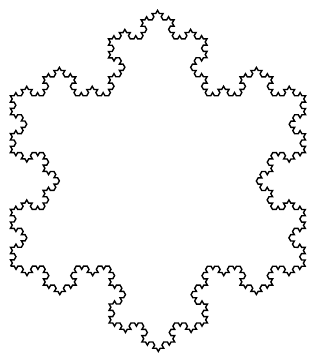}
 \caption{Fractal 
 (closed) 
 snowflake domain 
  $ %\overline
  {\Omega}$.}
  \label{fig-snow}
 \end{figure}
 
A central tool in our analysis is the expression
\[\E_{\mathcal{A}}(f,f)=\int_{\Omega} (A (x)\cdot\nabla f(x))\cdot\nabla f(x)\mathcal{L}^2(dx)+c_0\E_{\partial\Omega}(f,f), \]
which, together with a suitable domain of definition, is a local Dirichlet form.
Here \A\ is a bounded measurable uniformly elliptic two-by-two matrix valued function, $\mathcal{L}^2$ denotes the usual two dimensional Lebesgue measure, $c_0>0$ is a fixed constant, and $\E_{\partial\Omega}$ denotes the usual Kusuoka-Kigami Dirichlet form on the snowflake boundary
$\partial\Omega$, which is a union of three copies of the classical Koch curve, \cite{Ki93, Ki01, Ku89, Ku93}. Using perturbation arguments we can pass from $\mathcal{E}_{\mathcal{A}}$ to the bilinear form corresponding to the parabolic Ventsell problem that we wish to study. This problem can be formulated as
\begin{equation}\label{e1}
\begin{cases}
u_t(t,x)-L_\mathcal{A} u(t,x)-\vec{b}(x)\cdot\nabla u(t,x)=f(t,x)& \   \text{\    in $(0,T]\times\Omega$} 
\\
u_t(t,x)-c_0\Delta_{\partial\Omega} u(t,x)-b_{\partial\Omega}(x) D_{\partial\Omega} u(t,x)+c(x)u(t,x)&\\
\ \text{\ \hskip11em} =-\dfrac{\partial u(t,x)}{\partial n_\A}+f(t,x) & \   \text{\   in $(0,T]\times \partial\Omega$} 
\\
u(0,x)=u_0(x) & \ \text{\   in $\Omega$}.\\
\end{cases}
\end{equation}

Here $L_\mathcal{A} u(t,\cdot)=\diverg(\mathcal{A}\cdot\nabla u(t,\cdot))$, and $\Delta_{\partial\Omega}$ is the Kusuoka-Kigami Laplacian on $\partial\Omega$. The function $f$ is a time-dependent external forcing, the function $c$ is a stationary scalar potential on $\partial\Omega$, and $u_0$ is a given initial condition. The term 
{\scriptsize $\dfrac{\partial u(t,\cdot)}{\partial n_\A}$} is the co-normal derivative of $u(t,\cdot)$ across $\partial\Omega$, to be defined in a suitable distributional sense, see for example \cite{LaV14} or Section \ref{S:co-normal}. The symbol $\vec{b}$ denotes a stationary drift vector field in $\Omega$, and $b_{\partial\Omega}$ is a drift vector field on the snowflake boundary $\partial\Omega$. This drift $b_{\partial\Omega}$ is defined using a special case of the intrinsic approach to derivatives for Dirichlet forms, see e.g. \cite{CS03, HRT13}, applied to the form $\mathcal{E}_{\partial\Omega}$. We point out that $b_{\partial\Omega}$ does not have to be the trace of $\vec{b}$ on $\partial\Omega$. The expression $D_{\partial\Omega}u(t,\cdot)$ plays the role of the tangential derivative of $u(t,\cdot)$ along $\partial\Omega$. We prove that it can be defined as the measurable gradient on the fractal boundary $\partial\Omega$ of the restriction $u(t,\cdot)|_{\partial\Omega}$ of $u(t,\cdot)$ to $\partial\Omega$. In principle this is nothing else but a version of the gradient in the sense of Kusuoka \cite{Ku93}, Kigami \cite{Ki93}, Strichartz \cite{Str00} and Teplyaev \cite{T00}. But only the more functional analytic description \cite{CS03, CS09, HRT13, IRT} permits sufficient flexibility for applications to partial differential equations \cite{HRT13, HTams}, magnetic operators \cite{H14b, HR14, HTc} or geometric analysis \cite{HKT15, HTams, HTcurl}. There are (locally harmonic) coordinate functions $h_i$ on the snowflake $\partial\Omega$, and for functions on $\partial \Omega$ locally representable as $F_i\circ h_i$ with bounded functions $F_i\in C^1(\mathbb{R})$ we locally have
\begin{equation}\label{E:composition}
D_{\partial\Omega}(F_i\circ h_i)(x)=F_i'(h_i(x)),
\end{equation}
i.e. the measurable gradient admits a pointwise bounded expression. In this case the measurable gradient can also be expressed as a limit of difference quotients in the spirit of Mosco \cite{Mo02, Mo04}, see Section \ref{S:tangential}. The span of such functions is a dense subspace of the domain of $\mathcal{E}_{\partial\Omega}$.

Abstract gradients can be defined for any Dirichlet form, \cite{CS03}, and this is already sufficient to formulate and study partial differential equations with first order terms on some abstract level, \cite{HRT13, HTams}. Interpretations of these abstract gradients as limits of difference quotients, \cite{Mo02, Mo04}, or in terms of finite element approximations, \cite{CS09, IRT}, are possible also for energy forms on other sufficiently structured fractals, such as for instance resistance forms on p.c.f. self-similar sets, \cite{Ki93, Ki01, T08}. In many cases one can find pointwise bounded coordinate representations similar to (\ref{E:composition}), what allows to consider non-linear equations such as in \cite{HRT13}.

The novelty of the present paper is that we prove that in our situation $D_{\partial\Omega}u(t,\cdot)$ is a well defined $L^2$-function on $\partial\Omega$, and so we can consider problem (\ref{e1}) for instance with a bounded measurable function  $b_{\partial\Omega}$ on the fractal $\partial\Omega$, which then plays the role of a (bounded measurable) tangential drift vector field. 
This is very special and only possible because $\partial\Omega$ is a curve. 
An interpretation of measurable gradients as measurable vector-valued functions will generally 
 be more involved for other more complicated fractal sets.
 However, such  interpretation  
 is very important for considering non-linear PDEs, 
 such as those in  \cite[Section 4]{HRT13} and 
 \cite[Sections  4, 5]{AkkermansMallick}.

 The Koch snowflake domain appears in  \cite{FLV95, LG97, L91, LNRG96, LP95,vdB}. 
 Our current paper is primarily motivated by physics considerations, see 
 \cite{AkkermansMallick} and, in broader terms, \cite{A,ADT,ADT2010PRL,ADTV,Dunne}, 
  and by  \cite{HTcurl}. An extensive discussion of potential applications and numerical implementations can be found in \cite{LaV14}. We would like to point out that, up to the behaviour at time $t=0$ problem (\ref{e1}) extends \cite[($\overline{P}$), page 682]{LaV14}, which correspond to the special case with $\vec{b}\equiv 0$, $b_{\partial\Omega}\equiv 0$ and $u_0=0$. 
   
 We note that the Koch curve is a Lindstr{\o}m's nested fractal, see \cite{Li}, and so the papers \cite{HK03,HinoK2006,Ku2000} provide some heat semigroup  
 %and large deviations 
 estimates. See especially  Theorem 2.11 in \cite{Ku2000}, where ultracontractivity bounds for the Brownian motion penetrating into the fractal are deduced from
 ultracontractivity bounds for the diffusion on the fractal. However, pointwise heat kernel estimates are not available in the existing literature, in particular because the reference measure on  \oO\ is not volume doubling, although  
  \cite{GK15} contains some results towards obtaining such heat kernel estimates. 

Our first main result is Theorem \ref{T:strictsol}, which claims that under standard conditions on $\mathcal{A}$, $\vec{b}$, $b_{\partial\Omega}$, $c$ and $f$ (such as ellipticity, form-boundedness, continuity and H\"older continuity) a suitable reformulation of (\ref{e1}) as an abstract Cauchy problem in an $L^2$-space has a unique classical solution.
As in \cite{LaV14} one can then pass to a rigorous reformulation of (\ref{e1}), see Section \ref{S:co-normal}, in particular (\ref{E:interpret}).

A second, auxiliary main result is an extension theorem for Lipschitz functions that complements known extension results,\cite{E1,E2}, and connects to coordinate methods, \cite{HKT15,HT15a}. The Dirichlet form $\mathcal{E}_{\partial\Omega}$ induces an intrinsic metric on $\partial\Omega$, cf. \cite{BM95, Mo94, Stu94, Stu95}. Theorem \ref{thm:lip} states that any $\mathcal{E}_{\partial\Omega}$-intrinsically Lipschitz function on $\partial\Omega$ has a Euclidean-Lipschitz extension to $\overline{\Omega}$. As a consequence, the Dirichlet form $\mathcal{E}_{\mathcal{A}}$ is shown to possess a core consisting of functions $u$ such that $u|_{\Omega}\in C^1(\Omega)$ and $u|_{\partial\Omega}$ is $\mathcal{E}_{\partial\Omega}$-intrinsically $C^1$ on $\partial\Omega$ in a suitable sense, Corollary \ref{cor:coord}. Moreover, one can find two coordinate functions $y_1,y_2$ 
   which are contained in the core and separate the points of  \oO. We believe that the coordinate approach will permit new regularity results for first order terms on $\partial\Omega$. For Theorem \ref{thm:lip} we provide two different proofs, which both proceed by triangulation methods and therefore are of potential interest for numerical implementations.

We proceed as follows. In Section \ref{enercurve} we recall the definition of the Kusuoka-Kigami energy form $\mathcal{E}_{\partial\Omega}$ and Laplacian $\Delta_{\partial\Omega}$ on the snowflake $\partial\Omega$. In Section \ref{S:tangential} we define the tangential gradient operator $D_{\partial\Omega}$ as an operator on the natural $L^2$-space on $\partial\Omega$. Section \ref{S:Cauchy} contains a set of standard conditions under which (\ref{e1}) admits a suitable reformulation as an $L^2$-Cauchy problem that has a unique classical solution. Following \cite{LaV14} we provide a concrete interpretation of this Cauchy problem as a version of (\ref{e1}) in Section \ref{S:co-normal}, the main issue here is a suitable explanation of the co-normal derivative. The extension theorem is discussed in Section \ref{S:extension} and the two proofs are provided in Subsections \ref{subsub:lip} and \ref{subsub:lip2}, respectively.

We agree to write $\mathcal{E}_{\mathcal{A}}(f):=\mathcal{E}_{\mathcal{A}}(f,f)$, and we do similarly whenever we insert the same argument twice in other bilinear expressions. 

\section{Energy and Laplacian on the Snowflake}\label{enercurve}

To fix notation we briefly recall the construction of the Koch curve. The Koch curve $K$ is the unique nonempty compact subset $K$ of $\mathbb{C}$ such that $K=\bigcup_{i=1}^4 \psi_i(K)$, where $\psi_1(z)=\frac{z}{3}$, $\psi_2(z)=\frac{z}{3}e^{i\frac{\pi}{3}}+\frac13$, $\psi_3(z)=\frac{z}{3}e^{-i\frac{\pi}{3}}+i\frac{\sqrt{3}}{6}$ and $\psi_4(z)=\frac{z+2}{3}$. The set $V_0(K)=\left\lbrace 0, 1\right\rbrace\subset \mathbb{C}$ is the set of essential fixed points of the iterated function system 
$\left\lbrace \psi_1, ..., \psi_4\right\rbrace$. For $j_1,...,j_n\in \left\lbrace 1,2,3,4\right\rbrace$ write $\psi_{j_1\cdots j_n}:=\psi_{j_1}\circ\cdots\circ\psi_{j_n}$ and $V_{j_1\cdots j_n}(K):=\psi_{j_1\cdots j_n}(V_0(K))$. For fixed $n$ we put \[V_n(K):=\bigcup_{j_1,\dots, j_n=1}^4 V_{j_1\cdots j_n}(K)\] 
and finally, set 
\[V_\ast(K):=\bigcup_{n\geq 0} V_n(K).\]

We recall the construction of the Kusuoka-Kigami energy form on the Koch curve $K$, \cite{FL04, Ki01}. Setting 
\begin{equation}\label{enu}
\E_K^{(n)}(u):=\frac{1}{2}\ 4^n \sum_{p\in V_n(K)}\sum_{q\sim_n p} (u(p)-u(q))^2, 
\end{equation}
we can define discrete energy forms $\E_K^{(n)}$ on the space of functions $u:V_\ast(K)\to \mathbb{R}$. The notation $q\sim_n p$ in (\ref{enu}) means that we sum over all $q$ that are $n$-neighbors of $p$. Due to the rescaling with prefactor $4^n$ the discrete energies $\E _K^{(n)}(u)$ are non-decreasing in $n$ and 
\begin{equation}\label{2bis}
\E_K(u):=\lim_{n \to \infty}\E_K^{(n)}(u)
\end{equation} 
defines a non-trivial quadratic form $\E_K$ with domain $\mathcal{D}_{\ast}(\mathcal{E}_K):=\{u:V_\ast(K)\rightarrow \mathbb{R}: \E_K (u)<\infty \}$. The metric
\[d_R(p,q):=\sup\left\lbrace\frac{|u(p)-u(q)|^2}{\mathcal{E}_K(u)}:\ \ u\in \mathcal{D}_{\ast}(\mathcal{E}_K), \ \mathcal{E}_K(u)>0\right\rbrace, \ \ p,q\in V_\ast(K),\]
on $V_\ast(K)$ extends to a metric on $K$. The metric space $(K, d_R)$ is compact, and the resistance topology (i.e. the topology induced by $d_R$) agrees with the Euclidean trace topology, \cite{Ba98, HW06}. The functions in $\mathcal{D}_{\ast}(\mathcal{E}_K)$ are seen to extend uniquely to functions on $K$ which are continuous in resistance topology, we denote these extensions by the same symbol. We write $\mathcal{D}(\mathcal{E}_K) :=\{u \in C(K): \mathcal{E}_K  (u)<\infty \}$, where $\mathcal{E}_K  (u):=\mathcal{E}_K (u|_{V_\ast(K)})$. By polarization this defines a resistance form $(\mathcal{E}_K,\mathcal{D}(\mathcal{E}_K))$ on $K$ in the sense of Kigami \cite{Ki01, Ki03, Ki12}.
For any $u\in\mathcal{D}(\mathcal{E}_K)$ and any $x,y\in K$ we have
\begin{equation}\label{E:resistest}
|u(x)-u(y)|\leq d_R(x,y)^{1/2}\mathcal{E}_K(u)^{1/2}.
\end{equation}
The resistance form $(\mathcal{E}_K,\mathcal{D}(\mathcal{E}_K))$ is regular, that is, $\mathcal{D}(\mathcal{E}_K)$ is uniformly dense in $C(K)$.
 For any $u,v\in\mathcal{D}(\mathcal{E}_K)$ we have
\[\E_K(u,v):=\lim_{n\to\infty}\frac{1}{2}\ 4^n \sum_{p\in
V_n(K)}\sum_{q\sim_n p}
(u(p)-u(q))(v(p)-v(q)).\]
A function $u$ on $K$ is said to be harmonic with boundary values $u(0)$ and $u(1)$ on $V_0(K)$ if it minimizes the discrete energies $\mathcal{E}_K^{(n)}$ for all $n\geq 1$. In this case $u\in \mathcal{D}(\mathcal{E}_K)$ and for all $n\geq 1$ we have
\[\mathcal{E}_K(u)=\mathcal{E}_K^{(n)}(u)=\mathcal{E}_K^{(0)}(u)=\frac12(u(1)-u(0))^2.\] 
In the sequel let $h$ denote the unique harmonic function on $K$ with boundary values $h(0)=0$ and $h(1)=1$. We use $h$ as a coordinate function on $K$. It provides a homeomorphism of $K$ onto $[0,1]$. If we let $\delta=\dim_H K=\frac{\ln 4}{\ln 3}$ denote the Hausdorff dimension of $K$, then $K$, equipped with the quasi-distance $d_{\delta}(x,y):=|x-y|^\delta$, $x,y\in K$, is a variational fractal in the sense of \cite{Mo97}, and for any $p,q\in V_n(K)$ with $p\sim_n q$ we have 
\begin{equation}\label{E:Moscoformula}
|h(p)-h(q)|=d_\delta(p,q)=|p-q|^\delta.
\end{equation}
The space
\[\mathcal{S}_K:=\left\lbrace F\circ h: F\in C^1(\mathbb{R})\right\rbrace\] 
is a dense subspace of $\mathcal{D}(\mathcal{E}_K)$, \cite[Theorem 7]{T08}. Let $\mu_K$ denote the energy measure of $h$, that is, the unique nonnegative and finite Radon measure on $K$ such that 
\begin{equation}\label{E:BH}
\mathcal{E}_K(h,uh)-\frac12\mathcal{E}_K(u,h^2)=\int_Ku\:d\mu_K, \ \ u\in \mathcal{D}(\mathcal{E}_K).
\end{equation}
The chain rule, \cite[Theorem 3.2.2]{FOT94}, now implies that 
\begin{equation}\label{E:chain}
\mathcal{E}_K(u)=\int_{K} (F'(h(x)))^2 \mu_K(dx),
\end{equation}
for any $u=F\circ h$ from $\mathcal{S}_K$. From (\ref{enu}) and a straightforward calculation (similarly as in the case of the unit interval) one can see that up to the multiplicative constant $\frac12$, the energy measure $\mu_K$ of $h$ is just the normalized self-similar Hausdorff measure on $K$. Clearly the measure is finite and non-atomic and has full support. Therefore the form $(\mathcal{E}_K,\mathcal{D}(\mathcal{E}_K))$ is a strongly local regular Dirichlet form on $L^2(K,\mu_K)$, see e.g. \cite[Theorem 8.4]{Ki12}.

We pass to the Dirichlet form on the Koch Snowflake $\partial\Omega$, \cite{FL04}. 
The latter can be obtained as the union of three congruent copies $K_1$, $K_2$ and $K_3$ of the Koch curve $K$, arranged as symbolically depicted in Figure \ref{F:copies123}. Alternatively it can be obtained as the union of the three 'shifted' copies $K_4$, $K_5$ and $K_6$, Figure \ref{F:copies456}. Let $\varphi_i$, $i=1,...,6$, be the six uniquely determined orientation preserving Euclidean motions such that each $\varphi_i$ maps $K_i$ to $K$ so that for each $i$ the boundary point of $K_i$ which in counterclockwise orientation of $\partial\Omega$ comes first is mapped to $0\in K$.

\begin{figure}
\begin{minipage}{0.45\textwidth}
\begin{center}
\begin{tikzpicture}[scale=0.45]
\draw[black, very thick] (0.3,0)--(3,0);
\draw[black, very thick] (3,0)--(4.5,2.598);
\draw[black, very thick] (6,0)--(4.5,2.598);
\draw[black, very thick] (6,0)--(8.7,0);
\node at (4.5,1) {$K_1$};
\begin{scope}[shift={(9,0)}, rotate=-120]
\draw[black, very thick] (0.3,0)--(3,0);
\draw[black, very thick] (3,0)--(4.5,2.598);
\draw[black, very thick] (6,0)--(4.5,2.598);
\draw[black, very thick] (6,0)--(8.7,0);
\node at (4.5,1) {$K_2$};
\end{scope}
\begin{scope}[shift={(4.5,-7.794)}, rotate=-240]
\draw[black, very thick] (0.3,0)--(3,0);
\draw[black, very thick] (3,0)--(4.5,2.598);
\draw[black, very thick] (6,0)--(4.5,2.598);
\draw[black, very thick] (6,0)--(8.7,0);
\node at (4.5,1) {$K_3$};
\end{scope}
\end{tikzpicture}
\end{center}
\caption{Using the copies $K_1$, $K_2$ and $K_3$.}\label{F:copies123}
\end{minipage}
\hfill
\begin{minipage}{0.45\textwidth}
\begin{center}
\begin{tikzpicture}[scale=0.45]
\begin{scope}[shift={(4.5,2.598)}, rotate=-60]
\draw[black, very thick] (0.3,0)--(3,0);
\draw[black, very thick] (3,0)--($(4.5, {sqrt(6.75)})$);
\draw[black, very thick] (6,0)--(4.5,2.598);
\draw[black, very thick] (6,0)--(8.7,0);
\node at (4.5,1) {$K_4$};
\end{scope}
\begin{scope}[shift={(9,-5.196)}, rotate=-180]
\draw[black, very thick] (0.3,0)--(3,0);
\draw[black, very thick] (3,0)--($(4.5, {sqrt(6.75)})$);
\draw[black, very thick] (6,0)--(4.5,2.598);
\draw[black, very thick] (6,0)--(8.7,0);
\node at (4.5,1) {$K_5$};
\end{scope}
\begin{scope}[shift={(0,-5.196)}, rotate=-300]
\draw[black, very thick] (0.3,0)--(3,0);
\draw[black, very thick] (3,0)--($(4.5, {sqrt(6.75)})$);
\draw[black, very thick] (6,0)--(4.5,2.598);
\draw[black, very thick] (6,0)--(8.7,0);
\node at (4.5,1) {$K_6$};
\end{scope}
\end{tikzpicture}
\end{center}
\caption{Using the copies $K_4$, $K_5$ and $K_6$.}\label{F:copies456}
\end{minipage}
\end{figure}

Set
\[\mathcal{D}(\mathcal{E}_{\partial\Omega}):=\left\lbrace u:\partial\Omega\to\mathbb{R}: u|_{K_i}\circ \varphi_i^{-1}\in\mathcal{D}(\mathcal{E}_K),\ i=1,...,6\right\rbrace\]
and define
\begin{equation}\label{E:energysnow}
\mathcal{E}_{\partial\Omega}(u):=\mathcal{E}_{K_1}(u|_{K_1})+\mathcal{E}_{K_2}(u|_{K_2})+\mathcal{E}_{K_3}(u|_{K_3}),\ \ u\in\mathcal{D}(\mathcal{E}_{\partial\Omega}).
\end{equation}
In (\ref{E:energysnow}) we could replace $K_1$, $K_2$ and $K_3$ by $K_4$, $K_5$ and $K_6$. As the $\varphi_i$ are just Euclidean motions, we notationally identify  $u|_{K_i}$ with $u|_{K_i}\circ \varphi_i^{-1}$ and also identify the measures $\mu_{K_i}$, defined similarly as in (\ref{E:BH}), with 
their image measures under the $\varphi_i^{-1}$. Note that on each copy $K_i$ of $K$ the function $h_i=h\circ \varphi_i$ is a harmonic function as $h$ above (so that in counterclockwise orientation the boundary point of $K_i$ with value zero comes first). We equip $\partial\Omega$ with the measure  
\[\mu:=\mu_{K_1}+\mu_{K_2}+\mu_{K_3}.\] 
Then the quadratic form $(\mathcal{E}_{\partial\Omega}, \mathcal{D}(\mathcal{E}_{\partial\Omega}))$ defines a strongly local Dirichlet form on $L^2(\partial\Omega,\mu)$. We refer to it as the \emph{(Kusuoka-Kigami) energy form} on the snowflake. By $(\Delta_{\partial\Omega}, \mathcal{D}(\Delta_{\partial\Omega}))$ we denote the infinitesimal generator of $(\mathcal{E}_{\partial\Omega}, \mathcal{D}(\mathcal{E}_{\partial\Omega}))$ in the usual $L^2$-sense. We refer to it as the \emph{(Kusuoka-Kigami) Laplacian} on the snowflake. It can also be viewed in the variational sense as an operator $\Delta_{\partial\Omega}$ from the energy domain $\mathcal{D}(\mathcal{E}_{\partial\Omega})$ into its dual $(\mathcal{D}(\mathcal{E}_{\partial\Omega}))'$ so that
\[\left\langle \Delta_{\partial\Omega}f, g\right\rangle_{((\mathcal{D}(\mathcal{E}_{\partial\Omega}))',\mathcal{D}(\mathcal{E}_{\partial\Omega}))}=-\mathcal{E}_{\partial\Omega}(f,g),\ \ f,g\in\mathcal{D}(\mathcal{E}_{\partial\Omega}),\]
where $\left\langle\cdot,\cdot\right\rangle_{((\mathcal{D}(\mathcal{E}_{\partial\Omega}))',\mathcal{D}(\mathcal{E}_{\partial\Omega}))}$ denotes the dual pairing between the spaces $(\mathcal{D}(\mathcal{E}_{\partial\Omega}))'$ and $\mathcal{D}(\mathcal{E}_{\partial\Omega})$. We will use this interpretation in Section \ref{S:co-normal}.

\section{Tangential gradients on the snowflake}\label{S:tangential}

In this section we introduce the notion of tangential derivative on the
boundary $\partial\Omega$ of $\Omega$ and we prove that it is a $L^2$-function.
This is achieved by following the seminal paper \cite{Mo02} (see also
\cite{Mo04}), where the (discrete) tangential derivative is defined as the
limit of a suitable differential quotient, and by using the harmonic
coordinate method developped in \cite{T08}. The construction may also be seen as a particularly simple example of a measurable Riemannian structure, \cite{Hino13}. 

Let 
\[\mathcal{S}_{\partial\Omega}:=\left\lbrace u:\partial\Omega\to\mathbb{R}: u|_{K_i}\in\mathcal{S}_{K_i}, i=1,...,6\right\rbrace\]
where $\mathcal{S}_{K_i}:=\left\lbrace F\circ h_i: F\in C^1(\mathbb{R})\right\rbrace$. In a sense, the following density result is a special case of \cite[Theorem 7]{T08}.

\begin{lemma}
The space $\mathcal{S}_{\partial\Omega}$ is dense in $\mathcal{D}(\mathcal{E}_{\partial\Omega})$. 
\end{lemma}

Here and below we make use of the fact that the collection of open sets $\big\lbrace \mathring{K}_i\big\rbrace_{i=1}^6$ provides an open cover of $\partial\Omega$. 

\begin{proof}
Given $i$ and $u\in\mathcal{D}(\mathcal{E}_{\partial\Omega})$ let $(u_i^{(n)})_n\subset \mathcal{S}_{K_i}$ be a sequence of functions such that  
\[\lim_n\mathcal{E}_{K_i}(u|_{K_i}-u_i^{(n)})=0\ \ \text{ and }\ \ \lim_n \big\|u|_{K_i}-u_i^{(n)}\big\|_{\sup}=0,\] 
by the resistance estimate (\ref{E:resistest}) such a sequence can always be found.
Now let $\left\lbrace \chi_i\right\rbrace_i\subset\mathcal{D}(\mathcal{E}_{\partial\Omega})$ be an energy finite partition of unity subordinate to the cover $\big\lbrace \mathring{K}_i\big\rbrace_i$.
 Extending all $\chi_i u|_{K_i}$ and $\chi_i u_i^{(n)}$ by zero and using \cite[Corollary I.3.3.2]{BH91} also
\begin{align}
\mathcal{E}_{\partial\Omega}(\chi_i u|_{K_i}-\chi_i u_i^{(n)})^{1/2}&=\mathcal{E}_{K_i}(\chi_i u|_{K_i}-\chi_i u_i^{(n)})^{1/2}\notag\\
&\leq \mathcal{E}_{K_i}(\chi_i)^{1/2}\big\|u|_{K_i}-u_i^{(n)}\big\|_{\sup}+\left\|\chi_i\right\|_{\sup}\mathcal{E}_{K_i}(u|_{K_i}-u_i^{(n)})^{1/2}\notag
\end{align}
is seen to tend to zero as $n$ goes to infinity. Consider the function $u^{(n)}:=\sum_{i=1}^6 \chi_i u_i^{(n)}$, clearly an element of $\mathcal{D}(\mathcal{E}_{\partial\Omega})$. Since
\begin{multline}
2\mathcal{E}_{\partial\Omega}(u-u^{(n)})\leq \sum_{i=1}^6\mathcal{E}_{\partial\Omega}(\chi_i u|_{K_i}-\chi_i u_i^{(n)})\notag\\
+\sum_{i=1}^6\sum_{j=i+1}^6 \mathcal{E}_{\partial\Omega}(\chi_i u|_{K_i}-\chi_i u_i^{(n)})^{1/2}\mathcal{E}_{\partial\Omega}(\chi_j u|_{K_j}-\chi_j u_j^{(n)})^{1/2},
\end{multline}
we have $\lim_n \mathcal{E}_{\partial\Omega}(u-u^{(n)})=0$.
\end{proof}

For convenience define $V_\ast(\partial\Omega):=\bigcup_{i=1}^6 V_\ast(K_i)$, where for any $K_i$ the set $V_\ast(K_i)$ is defined in the same way as $V_\ast(K)$ for $K$. 

\begin{lemma}\label{L:agree}
Let $u\in\mathcal{S}_{\partial\Omega}$. Then for any $p\in V_\ast(\partial\Omega)$ the limit 
\begin{equation}\label{E:limit}
D_{\partial\Omega}u(p):=\lim_{V_\ast(\partial\Omega) \ni q\to p}\frac{u(p)-u(q)}{|p-q|^\delta},
\end{equation}
exists uniformly in $p\in V_\ast(\partial\Omega)$. If $p\in K_i$ and $u|_{K_i}=F_i\circ h_i$, then $D_{\partial\Omega}u(p)=F_i'(h_i(p))$.
Moreover, if $p\in K_i\cap K_j\cap V_\ast(\partial\Omega)$ and $u|_{K_j}=F_j\circ h_j$ then $F'_i(h_i(p))=F_j'(h_j(p))$.
\end{lemma}
\begin{proof}
Let $i$ be such that $K_i$ contains $p$ and $F_i\in C^1(\mathbb{R})$ such that $u|_{K_i}=F_i\circ h_i$. Given $\varepsilon>0$, choose $n$ such that $|F_i'(t)-F_i'(s)|<\varepsilon$ provided $|t-s|\leq 4^{-n}$. Then if $q\in K_i\cap V_n$ is such that $h_i(q)<h_i(p)$, we have 
\[\left|F'(h_i(p))-\frac{u(p)-u(q)}{|p-q|^\delta}\right|\leq \frac{1}{|h_i(p)-h_i(q)|}\int_{h_i(q)}^{h_i(p)}|F'(h_i(p))-F'(t)|dt<\varepsilon,\]
and similarly if $h_i(p)<h_i(q)$. The last statement follows because the right hand side in (\ref{E:limit}) does not depend on $i$.
\end{proof}

\begin{definition}
Given $x\in \mathring{K}_i$ and $u\in\mathcal{S}_{\partial\Omega}$ such that $u|_{K_i}=F_i\circ h_i$, we set
\begin{equation}\label{E:tangential}
D_{\partial\Omega}u(x):=F_i'(h_i(x)).
\end{equation}
We refer to $D_{\partial\Omega}u(x)$ as the \emph{tangential gradient} of $u$ along $\partial\Omega$ in $x\in \partial\Omega$.
\end{definition}

\begin{corollary}
The definition (\ref{E:tangential}) is independent of the particular choice of $h_i$ and $F_i$.
\end{corollary}
\begin{proof}
Assume that $x\in \mathring{K}_i\cap\mathring{K}_j$ and that we have two representations $u|_{K_i}=F_i\circ h_i$ and $u|_{K_j}=F_j\circ h_j$. Let $(p_n)_n\subset V_\ast(\partial\Omega)\cap \mathring{K}_i\cap\mathring{K}_j$ be a sequence with limit $x$. The functions $F_i'\circ h_i$ and $F'_j\circ h_j$ are (uniformly) continuous on $K_i\cap K_j$. Therefore 
\[F'_j(h_j(x))=\lim_n F_j(h_j(p))=\lim_n F_i(h_i(p))=F_i'(h_i(x))\]
by Lemma \ref{L:agree}.
\end{proof}

The tangential gradient can be interpreted as a linear operator.

\begin{corollary}
The tangential gradient $D_{\partial\Omega}$ defines a bounded linear operator $D_{\partial\Omega}:\mathcal{S}_{\partial\Omega}\to L^2(\partial\Omega,\mu)$, more precisely, we have 
\[\left\|D_{\partial\Omega}u\right\|_{L^2(\partial\Omega,\mu)}^2=\mathcal{E}_{\partial\Omega}(u)\ \ \text{ for all }\  u\in\mathcal{S}_{\partial\Omega}.\]
\end{corollary}
\begin{proof}
Suppose $F_i\in C^1(\mathbb{R})$, $i=1,2,3$ are such that $u|_{K_i}=F_i\circ h_i$. Then by (\ref{E:chain}) and (\ref{E:energysnow}) we have
\[\mathcal{E}_{\partial\Omega}(u)=\int_{K_1} (F'_1(h_1(x)))^2\mu_{K_1}(dx)+\int_{K_2} (F'_2(h_2(x)))^2\mu_{K_2}(dx)+\int_{K_3} (F'_3(h_3(x)))^2\mu_{K_3}(dx),\]
what yields the result.
\end{proof}

The closedness of $(\mathcal{E}_{\partial\Omega},\mathcal{D}(\mathcal{E}_{\partial\Omega}))$ implies the following.

\begin{proposition}\label{P:closedop}
The operator $D_{\partial\Omega}$ extends to a closed unbounded operator $D_{\partial\Omega}:L^2(\partial\Omega,\mu)\to L^2(\partial\Omega,\mu)$ with domain $\mathcal{D}(\mathcal{E}_{\partial\Omega})$. Moreover, for any $u\in \mathcal{D}(\mathcal{E}_{\partial\Omega})$ we have 
\[\mathcal{E}_{\partial\Omega}(u)=\left\|D_{\partial\Omega}u\right\|_{L^2(\partial\Omega,\mu)}^2\]
and the energy measure $\nu_u$ of $u$ is absolutely continuous with respect to $\mu$ with density
\begin{equation}\label{E:Gammadomega}
\Gamma_{\partial\Omega}(u)=(D_{\partial\Omega}u)^2\ \ \text{$\mu$-a.e.}
\end{equation}
\end{proposition}

\begin{proof}
The first statement and the identity for the energy follow from the closedness of $(\mathcal{E}_{\partial\Omega},\mathcal{D}(\mathcal{E}_{\partial\Omega}))$. The last identity is immediate for all functions from $\mathcal{S}_{\partial\Omega}$. By polarization we obtain a bilinear map $(u,v)\mapsto \Gamma_{\partial\Omega}(u,v)=D_{\partial\Omega}uD_{\partial\Omega}v$ on $\mathcal{S}_{\partial\Omega}$. If now $(u_n)_n\subset \mathcal{S}_{\partial\Omega}$ is a sequence converging to $u\in\mathcal{D}(\mathcal{E}_{\partial\Omega})$ in $\mathcal{D}(\mathcal{E}_{\partial\Omega})$, then the limit
\[\lim_n \Gamma_{\partial\Omega}(u_n)=\Gamma_{\partial\Omega}(u)\]
exists in $L^1(\partial\Omega,\mu)$, as follows from $|\Gamma_{\partial\Omega}(u_n)^{1/2}-\Gamma_{\partial\Omega}(u_m)^{1/2}|\leq \Gamma_{\partial\Omega}(u_n-u_m)^{1/2}$ (which is a simple consequence of the bilinearity of $\Gamma_{\partial\Omega}$) together with the straightforward estimate
\[\int_{\partial\Omega}|\Gamma_{\partial\Omega}(u_n)-\Gamma_{\partial\Omega}(u_m)|d\mu\leq \left(\int_{\partial\Omega}
|\Gamma_{\partial\Omega}(u_n)^{1/2}-\Gamma_{\partial\Omega}(u_m)^{1/2}|^2d\mu\right)^{1/2}
(8\sup_n\mathcal{E}_{\partial\Omega}(u_n))^{1/2}.\]
\end{proof}

\begin{remark}\mbox{}
\begin{enumerate}
\item[(i)] According to general theory the adjoint $D_{\partial\Omega}^\ast$ of $D_{\partial\Omega}$ is a densely defined and closed linear operator from $L^2(\partial\Omega,\mu)$ into itself. The image $\im\:D_{\partial\Omega}$ of $D_{\partial\Omega}$ is a closed subspace of $L^2(\partial\Omega,\mu)$, see \cite[Section 4]{HKT15}, and we observe the orthogonal decomposition $L^2(\partial\Omega,\mu)=\im D_{\partial\Omega}\oplus \ker D_{\partial\Omega}^\ast$. Since $\partial\Omega$ is finitely ramified, it follows from \cite[Theorem 5.6]{IRT} that $\ker D_{\partial\Omega}^\ast$ is one-dimensional. The snowflake boundary $\partial\Omega$ is homeomorphic to the unit circle, and as in the case of the classical energy form on the unit circle $\ker D_{\partial\Omega}^\ast$ consists of the constant functions.
\item[(ii)] As already mentioned in the introduction, an analog of $D_{\partial\Omega}$, namely an (abstract) gradient, can be defined for any Dirichlet form, \cite{CS03}, and for many energy forms on fractals it admits very intuitive discrete approximations \cite{CS09, IRT, Ki93, Ku93, Mo02, Mo04, Str00, T00, T08}. In general such gradients take functions into (abstract) generalized $L^2$-vector fields, see for instance \cite{HT15a}. This is already sufficient to study partial differential equations on a certain functional analytic level, and we would like to emphasize that for many purposes it is not even be needed that energy measures are absolutely continuous. The specific (and by no means general) feature of the present situation is that the gradient of a function can again be seen as a function.
\end{enumerate}
\end{remark}

\section{Quadratic forms, semigroups and Cauchy problems}\label{S:Cauchy}

Now consider the function space 
\[V(\Omega,\partial\Omega):=\left\lbrace u\in H^1(\Omega): u|_F\in\mathcal{D}(\mathcal{E}_{\partial\Omega})\right\rbrace.\]
This space is well defined and non-empty, see e.g. \cite[Propositions 4.8 and 4.9]{La02} or \cite{La03}, and equipped with the scalar product
\begin{equation}\label{E:Hilbert}
\left\langle u,v\right\rangle_{V(\Omega,\partial\Omega)}=\left\langle u,v\right\rangle_{H^1(\Omega)}+\mathcal{E}_{\partial\Omega}(u|_{\partial\Omega}, v|_{\partial\Omega})+\left\langle u|_{\partial\Omega},v|_{\partial\Omega}\right\rangle_{L^2(\partial\Omega,\mu)},
\end{equation}
it is a Hilbert space, \cite[Proposition 3.2]{LaV14}. The last statement is a consequence of known trace and extension results, \cite{JW84, Tri97, W91}, a short survey is provided in Appendix \ref{App:A}.

Let
\[m=\mathcal{L}^2|_{\Omega}+\mu\]
denote the sum of the measures $\mathcal{L}^2|_{\Omega}$ and $\mu$. For any $c_0>0$
we can define a quadratic form $(\mathcal{E}_0, V(\Omega,\partial\Omega))$ on $L^2(\overline{\Omega}, m)$ by 
\[\mathcal{E}_0(u):=\int_\Omega (\nabla u)^2 d\mathcal{L}^2+c_0\mathcal{E}_{\partial\Omega}(u|_{\partial\Omega}),\ \ u\in V(\Omega,\partial\Omega), \]
and its polarized version 
\begin{equation}\label{E:unperturbed}
\mathcal{E}_0(u,v):=\int_\Omega \nabla u\cdot\nabla v\:d\mathcal{L}^2+c_0\mathcal{E}_{\partial\Omega}(u|_{\partial\Omega}, v|_{\partial\Omega}), \ \ u,v\in V(\Omega,\partial\Omega),
\end{equation}
yields a Dirichlet form $(\mathcal{E}_0, V(\Omega,\partial\Omega))$ on $L^2(\overline{\Omega}, m)$.

\begin{assumption}\label{A:conditions}
Now let us assume we are given the following data:
\begin{itemize}
\item[(i)] a coefficient matrix $\mathcal{A}=(a_{ij})_{i,j=1}^2$ that is symmetric, $a_{ji}=a_{ij}$, has bounded entries $a_{ij}\in L^\infty(\Omega)$, and satisfies the ellipticity condition
\[\sum_{i,j=1}^2 a_{ij}\xi_i\xi_j\geq \lambda \sum_{i=1}^2|\xi_i|^2,\ \ \xi=(\xi_1,\xi_2)\in\Omega,\]
with a constant $\lambda>0$,
\item[(ii)] a vector field $\vec{b}=(b_1,b_2)\in L^2(\Omega,\mathbb{R}^2)$ such that 
\[\int_{\Omega}u^2 (\vec{b}\cdot\vec{b})\:d\mathcal{L}^2\leq\gamma_1\int_\Omega (\nabla u)^2d\mathcal{L}^2+\gamma_2\big\|u\big\|_{L^2(\Omega)}^2, \ \ u\in V(\Omega,\partial\Omega),\]
with positive constants $\gamma_1$ and $\gamma_2$ such that $\sqrt{2\gamma_1}<\lambda$,
\item[(iii)] a 'boundary vector field' $b_{\partial\Omega}\in L^2(\partial\Omega,\mu)$ such that 
\[\int_{\partial\Omega} u^2b_{\partial\Omega}^2\:d\mu\leq \delta_1\mathcal{E}_{\partial\Omega}(u|_{\partial\Omega})+\delta_2\left\|u\right\|_{L^2(\partial\Omega,\mu)}^2, \ \ u\in V(\Omega,\partial\Omega),\]
with positive constants $\delta_1$ and $\delta_2$ such that $\sqrt{2\delta_1}<c_0$,
\item[(iv)] a continuous function $c$ on $\partial\Omega$.
\end{itemize}
\end{assumption}

Given data as in Assumption \ref{A:conditions} consider the bilinear form defined by
\begin{multline}
\mathcal{E}(u,v)=\int_{\Omega}(\mathcal{A}(x)\cdot\nabla u(x))\cdot\nabla v(x)\mathcal{L}^2(dx)-\int_{\Omega}(\vec{b}(x)\cdot\nabla u(x))v(x)\mathcal{L}^2(dx)\notag\\
+c_0\mathcal{E}_{\partial\Omega}(u,v)-\int_{\partial\Omega}b_{\partial\Omega}(x)D_{\partial\Omega}u(x)v(x)\mu(dx)+\int_{\partial\Omega}c(x)u(x)v(x)m(dx),
\end{multline}
$u,v\in V(\Omega,\partial\Omega)$. Given $\alpha >0$ we write 
\[\mathcal{E}_\alpha(u,v):=\mathcal{E}(u,v)+\alpha\left\langle u,v\right\rangle_{L^2(\overline{\Omega},m)}.\]

\begin{proposition}
Let $\mathcal{A}$, $\vec{b}$, $b_{\partial\Omega}$ and $c$ be as in Assumption \ref{A:conditions}. Then $(\mathcal{E}, V(\Omega,\partial\Omega))$ is a closed coercive form in the wide sense, i.e. there is some $\alpha>0$ such that $u\mapsto \mathcal{E}_\alpha(u,u)$ defines a positive definite closed quadratic form on $L^2(\overline{\Omega},m)$ with domain $V(\Omega,\partial\Omega)$ and
\begin{equation}\label{E:sector}
|\mathcal{E}_{\alpha+1}(u,v)|\leq K\:\mathcal{E}_{\alpha+1}(u)^{1/2}\mathcal{E}_{\alpha+1}(v)^{1/2},\ \ u,v\in V(\Omega,\partial\Omega),
\end{equation}
with a universal constant $K>0$.
\end{proposition}

\begin{proof} By (i) the form
\begin{equation}\label{E:EA}
\mathcal{E}_{\mathcal{A}}(u,v):=\int_{\Omega}\mathcal{A}\cdot\nabla u\cdot\nabla v\:d\mathcal{L}^2+c_0\mathcal{E}_{\partial\Omega}(u|_{\partial\Omega}, v|_{\partial\Omega})
\end{equation}
defines a symmetric Dirichlet form $(\mathcal{E}_{\mathcal{A}}, V(\Omega,\partial\Omega))$ on $L^2(\overline{\Omega},m)$. Conditions (ii) and (iii) are form boundedness conditions, cf. \cite{RS80}. A straightforward calculation using (ii) shows that for any $\varepsilon>0$ we have
\begin{equation}\label{E:formbound}
\left|\int_{\Omega}(\vec{b}\cdot \nabla u)u\:d\mathcal{L}^2\right|\leq \frac{\sqrt{2\gamma_1}+\varepsilon^2}{\lambda}\int_{\Omega} (\mathcal{A}\cdot\nabla u)\cdot\nabla u\:d\mathcal{L}^2+\frac{2\gamma_2}{\varepsilon^2}\left\|u\right\|_{L^2(\overline{\Omega},m)}^2,
\end{equation}
and by (ii) we can find a small $\varepsilon$ such that $\sqrt{2\gamma_1}+\varepsilon^2<\lambda$. A similar bound follows from (iii), and we can find some $0<\varepsilon_0<1$ and $C>0$ such that 
\[|\widetilde{\mathcal{E}}(u,u)-\mathcal{E}_{\mathcal{A}}(u,u)|\leq \varepsilon_0\mathcal{E}_{\mathcal{A}}(u,u)+C\left\|u\right\|_{L^2(\overline{\Omega},m)}^2\]
for any $u\in V(\Omega, \partial\Omega)$, where $\widetilde{\mathcal{E}}(u,v)=\frac12(\mathcal{E}(u,v)+\mathcal{E}(v,u))$ denotes the symmetric part of $\mathcal{E}$. By \cite[Theorem X.17]{RS80} the symmetric form $(\widetilde{\mathcal{E}}, V(\Omega,\partial\Omega))$ is lower bounded and closed.
The sector condition (\ref{E:sector}) is straightforward.
\end{proof}

As a consequence $(\mathcal{E}, V(\Omega,\partial\Omega))$ generates an analytic semigroup $(T_t)_{t\geq 0}$ on $L^2(\overline{\Omega}, m)$ which satisfies \[\left\|T_t \right\|_{L^2(\overline{\Omega},m)\to L^2(\overline{\Omega},m)}\leq e^{\alpha t},\ \ t\geq 0.\] 
See for instance \cite[Corollary I.2.21]{MR92}.

Let $(A,\mathcal{D}(A))$ denote the $L^2(\overline{\Omega},m)$-generator of the form $(\mathcal{E},V(\Omega,\partial\Omega))$. Given $T>0$ and functions $f:[0,T]\to L^2(\overline{\Omega},m)$ and $u_0\in L^2(\overline{\Omega},m)$,
we consider the abstract Cauchy problem 
\begin{equation}\label{E:Cauchy}
\begin{cases}
\frac{du(t)}{dt}=Au(t)+f(t),\ \ 0< t\leq T,\\
u(0)=u_0
\end{cases}
\end{equation}
in $L^2(\overline{\Omega},m)$. A function $u\in C([0,T], L^2(\overline{\Omega},m))$ is called a classical solution to (\ref{E:Cauchy}) if $u\in C^1((0,T], L^2(\overline{\Omega},m))\cap C((0,T],\mathcal{D}(A))$ and $u$ satisfies (\ref{E:Cauchy}). Under Assumptions \ref{A:conditions} and a H\"older condition on $f$ the existence and uniqueness of a classical solution to (\ref{E:Cauchy}) follows from \cite[Theorem 4.3.1]{Lu95}.

\begin{theorem}\label{T:strictsol}
Suppose $0<\theta<1$ and $f\in C^\theta((0,T], L^2(\overline{\Omega},m))$ and that Assumption \ref{A:conditions} is satisfied. Then (\ref{E:Cauchy}) has a unique classical solution $u$, given by
\[u(t)=T_tu_0+\int_0^tT_{t-s}f(s)ds,\ \ 0\leq t\leq T.\]
\end{theorem}

\begin{remark}\mbox{}
\begin{itemize}
\item[(i)] According to \cite[Theorem 4.3.1]{Lu95} for any $\varepsilon>0$ we also have 
\[\left\|u\right\|_{C^1([\varepsilon,T], L^2(\overline{\Omega},m)}+\left\|u\right\|_{C([\varepsilon,T],\mathcal{D}(A))}\leq c_{\varepsilon}\left\|f\right\|_{C^\theta([\varepsilon,T], L^2(\overline{\Omega},m)}\]
for the classical solution $u$ of (\ref{E:Cauchy}). Here $c_{\varepsilon}>0$ is a constant depending only on $\varepsilon$.
\item[(ii)] If in addition $u_0\in \mathcal{D}(\mathcal{A})$ then by \cite[Theorem 4.3.1]{Lu95} the unique classical solution $u$ to (\ref{E:Cauchy}) as in Theorem \ref{T:strictsol} is also the unique strict solution to the problem 
\begin{equation}\label{E:Cauchystrict}
\begin{cases}
\frac{du(t)}{dt}=Au(t)+f(t),\ \ 0\leq t\leq T,\\
u(0)=u_0,
\end{cases}
\end{equation}
i.e. the unique function $u\in C([0,T], L^2(\overline{\Omega},m))$ such that $u\in C^1([0,T], L^2(\overline{\Omega},m))\cap C([0,T],\mathcal{D}(A))$ and  (\ref{E:Cauchystrict}) is satisfied. In this case the estimate in (i) also holds for $\varepsilon=0$.
\end{itemize}
\end{remark}

We point out that by (\ref{E:formbound}) and a similar bound for $\mathcal{E}_{\partial\Omega}$ there is a constant $C'>0$ such 
that for the form $\mathcal{E}_{\mathcal{A}}$ in (\ref{E:EA}) we have $\mathcal{E}_{\mathcal{A}}(g)\leq C'\mathcal{E}(g)$, $g\in V(\Omega,\partial\Omega)$. Together with Assumption \ref{A:conditions} (i) this implies that 
\begin{equation}\label{E:continuity}
\int_\Omega (\nabla g)^2\:d\mathcal{L}^2\leq \frac{C'}{\lambda}\:\mathcal{E}_1(g)\leq \frac{C'}{\lambda}\:(\left\|Ag\right\|_{L^2(\overline{\Omega},m)}+\left\|g\right\|_{L^2(\overline{\Omega},m)})^2
\end{equation}
for any $g\in\mathcal{D}(A)$.

\section{Strong interpretation and co-normal derivatives}\label{S:co-normal}

Following the method in \cite[Section 6]{LaV14} we can reinterpret the abstract Cauchy problem (\ref{E:Cauchy}) as a rigorous version of (\ref{e1}). For the convenience of the reader we sketch the arguments.

Consider the space
\[V(\Omega)=\left\lbrace g\in H^1(\Omega): L_{\mathcal{A}}g\in L^2(\Omega)\right\rbrace,\]
where $L_{\mathcal{A}}g=\diverg(\mathcal{A}\nabla g)$. We equip it with the norm
\[\left\|g\right\|_{V(\Omega)}:=\left\|L_{\mathcal{A}}g\right\|_{L^2(\Omega)}+\left\|\nabla g\right\|_{L^2(\Omega,\mathbb{R}^2)}+\left\|g\right\|_{L^2(\Omega)}.\]
Proceeding as in \cite[Theorem 4.15]{La02} and \cite{LaV13} we can see that for any $g\in V(\Omega)$ we can define a distribution $l_g\in (H^1(\Omega))'$ by 
\[l_g(v):=\int_{\Omega}(\mathcal{A}\cdot\nabla g)\cdot\nabla v\:d\mathcal{L}^2+\int_{\Omega}vL_{\mathcal{A}}g\:d\mathcal{L}^2, \ \ v\in H^1(\Omega).\]
By Cauchy-Schwarz there is a constant $c>0$ such that for any $g\in V(\Omega)$ and $v\in H^1(\Omega)$ we have
\begin{equation}\label{E:lg}
|l_g(v)|\leq c\left\|g\right\|_{V(\Omega)}\left\|v\right\|_{H^1(\Omega)}
\end{equation}
and moreover, if $v, v'\in H^1(\Omega)$ are such that $v|_{\partial\Omega}=v'|_{\partial\Omega}$ then $l_g(v)=l_g(v')$. In conjunction with the trace and extension results mentioned in Appendix \ref{App:A} we can therefore view each $l_g$ as a distribution $\frac{\partial g}{\partial n_{\mathcal{A}}}$ in $(B^{2,2}_{\delta/2}(\partial\Omega))'$ and use the notation
\[\big\langle \frac{\partial g}{\partial n_{\mathcal{A}}}, v|_{\partial\Omega}\big\rangle_{(B^{2,2}_{\delta/2}(\partial\Omega))',B^{2,2}_{\delta/2}(\partial\Omega))}=l_g(v),\ \ v\in H^1(\Omega)\]
in terms of a dual pairing. See also \cite[formula (3.26)]{La03} and \cite{LaV14} for a similar reasoning. By the boundedness of the extension operator $\Ext: B^{2,2}_{\delta/2}(\partial\Omega)\to H^1(\Omega)$ together with estimate (\ref{E:lg}) there is a constant $c>0$ such that for any $g\in V(\Omega)$ we have
\begin{equation}\label{E:dualnorm}
\big\|\frac{\partial g}{\partial n_{\mathcal{A}}}\big\|_{(B^{2,2}_{\delta/2}(\partial\Omega))'}\leq c\:\left\|g\right\|_{V(\Omega)}.
\end{equation}

Given $f\in C^\theta([0,T], L^2(\overline{\Omega},m))$ as in Theorem \ref{T:strictsol} let now $u\in C((0,T],\mathcal{D}(A))$ denote the unique classical solution to (\ref{E:Cauchy}). Testing with functions from $C_c^\infty(\Omega)$ 
we observe that for any $t\in (0,T]$ we have
\[u_t(t,x)=L_{\mathcal{A}}u(t,x)+\vec{b}(x)\cdot\nabla u(t,x) +f(t,x)\]
in the distributional sense. By density this also holds in $L^2(\Omega)$. Due to the hypothesis on $f$ and according to Theorem \ref{T:strictsol}, the function $t\mapsto u_t(t)+f(t)$ is an element of $C((0,T], L^2(\Omega))$. Since 
\[\left\|\nabla u(t)-\nabla u(s)\right\|_{L^2(\Omega)}\leq \frac{C'}{\lambda}\left(\left\|A(u(t)-u(s))\right\|_{L^2(\overline{\Omega},m)}+\left\|u(t)-u(s)\right\|_{L^2(\overline{\Omega},m)}\right)\]
for any $s,t\in (0,T]$ by (\ref{E:continuity}), Theorem \ref{T:strictsol} implies that also $t\mapsto \vec{b}\cdot\nabla u(t)$ is in $C((0,T],L^2(\Omega))$. Consequently also $L_{\mathcal{A}}u\in C((0,T],L^2(\Omega))$ and therefore $u\in C((0,T], V(\Omega))$.

If for fixed $t\in (0,T]$ (suppressed from notation) we multiply the first identity in (\ref{E:Cauchy}) with some $\varphi\in V(\Omega,\partial\Omega)$, we obtain 
\begin{multline}
\left\langle u_t,\varphi\right\rangle_{L^2(\Omega)}+\left\langle u_t,\varphi\right\rangle_{L^2(\partial\Omega,\mu)}\notag\\
=-\int_{\Omega}(\mathcal{A}\cdot\nabla u)\cdot\nabla \varphi\:d\mathcal{L}^2+\int_{\Omega}(\vec{b}\cdot\nabla u)\varphi\:d\mathcal{L}^2
-c_0\mathcal{E}_{\partial\Omega}(u|_{\partial\Omega},\varphi|_{\partial\Omega})
\notag\\
+\int_{\partial\Omega}(b_{\partial\Omega}D_{\partial\Omega}u|_{\partial\Omega})\varphi d\mu-\left\langle c u,\varphi \right\rangle_{L^2(\partial\Omega,\mu)}.
\end{multline}
Taking into account that
\[\int_{\Omega}(\mathcal{A}\cdot\nabla u)\cdot\nabla\varphi\:d\mathcal{L}^2=-\big\langle \frac{\partial u}{\partial n_{\mathcal{A}}}, \varphi|_{\partial\Omega}\big\rangle_{(B^{2,2}_{\delta/2}(\partial\Omega))',B^{2,2}_{\delta/2}(\partial\Omega))}+\int_{\Omega}\varphi L_{\mathcal{A}}u\:d\mathcal{L}^2\]
and noting that by the embedding of $\mathcal{D}(\mathcal{E}_{\partial\Omega})$ into $B^{2,2}_{\delta/2}(\partial\Omega)$, we may view $\frac{\partial u}{\partial n_{\mathcal{A}}}(t)$ as an element of $(\mathcal{D}(\mathcal{E}_{\partial\Omega}))'$, we find that  
\begin{equation}\label{E:interpret}
\begin{cases}
u_t(u)-L_{\mathcal{A}}u(t)-\vec{b}\cdot\nabla u(t) &=f(t)\ \ \text{in $L^2(\Omega)$},\ \ t\in(0,T],\\
u_t(t)-c_0\Delta_{\partial\Omega}u(t)|_{\partial\Omega}-b_{\partial\Omega}D_{\partial\Omega}u(t)|_{\partial\Omega}+cu(t)&=-\frac{\partial u}{\partial n_{\mathcal{A}}}+f(t) \ \ \text{in $(\mathcal{D}(\mathcal{E}_{\partial\Omega}))'$},\ \ t\in (0,T],\\
u(0)&=u_0\ \ \text{ in $L^2(\overline{\Omega},m)$}.
\end{cases}
\end{equation}
Problem (\ref{E:interpret}) is a rigorous version of (\ref{e1}), and we may regard $u$ as in Theorem \ref{T:strictsol} as a solution. See \cite{LaV14} for further related details and references.

\section{An extension principle}\label{S:extension}

One of the main difficulties in Dirichlet form vector analysis on \oO\ is the fact that $\Omega$ is not a Lipschitz domain. Nevertheless we can prove Theorem \ref{thm:lip} below, which is an extension theorem that in some sense augments the results of \cite{E1,E2}. It also allows to use the coordinate approach from \cite{HKT15,HT15a}.    

The $\mathcal{E}_{\partial\Omega}$-intrinsic distance of two points $x,y\in\partial\Omega$ is defined as 
\[d_{\mathcal{E}_{\partial\Omega}}(x,y):=\sup\left\lbrace u(x)-u(y): u\in\mathcal{D}(\mathcal{E}_{\partial\Omega})\ \text{ with $\Gamma_{\partial\Omega}(u)\leq 1$ $\mu$-a.e.}\right\rbrace,\]
where for any $u\in\mathcal{D}(\mathcal{E}_{\partial\Omega})$, the function $\Gamma_{\partial\Omega}(u)=(D_{\partial\Omega}u)^2$ as in (\ref{E:Gammadomega}) is the density of the energy measure $\nu_u$ of $u$ with respect to $\mu$. Because $\partial\Omega$ is compact in resistance topology and $\mathcal{D}(\mathcal{E}_{\partial\Omega})\subset C(\partial\Omega)$, this is consistent with the usual definition of intrinsic distance in terms of Dirichlet forms \cite{BM95, Mo94, Stu94, Stu95}.  We record the following observation.

\begin{corollary}\label{C:identifymetric}
For any $i$ and any $x,y\in\mathring{K}_i$ we have $d_{\mathcal{E}_{\partial\Omega}}(x,y)=|h_i(x)-h_i(y)|$. 
\end{corollary}

\begin{proof}
Consider a function $u$ which equals $h_i$ on $K_i$, $1-h_j$ on one of the neighboring $K_j$ and is constant outside $K_i\cup K_j$. Then clearly $u\in\mathcal{D}(\mathcal{E}_{\partial\Omega})$ and $\Gamma(u)\leq 1$ $\mu$-a.e. Hence 
\[|h_i(x)-h_i(y)|=|u(x)-u(y)|\leq d_{\mathcal{E}_{\partial\Omega}}(x,y).\]

To see the converse inequality let $\varepsilon>0$ and let $u\in\mathcal{D}(\mathcal{E}_{\partial\Omega})$ be a function with $\Gamma(u)\leq 1$ $\mu$-a.e. and such that $d_{\mathcal{E}_{\partial\Omega}}(x,y)\leq u(x)-u(y)+\frac{\varepsilon}{3}$. The desired inequality follows provided we can find $v\in\mathcal{D}(\mathcal{E}_{\partial\Omega})$ with $v=F_i\circ h_i$ and $\left\|F_i'\right\|_{\sup}\leq 1$ such that $\sup_{x\in K_i}|u(x)-v(x)|<\frac{\varepsilon}{3}$, because in this case $d_{\partial\Omega}(x,y)\leq v(x)-v(y)+\varepsilon$. By the resistance estimate (\ref{E:resistest}) it is sufficient to show that for any $\delta>0$ we can find a function $v\in\mathcal{D}(\mathcal{E}_{\partial\Omega})$ with $v=F_i\circ h_i$ and $\left\|F_i'\right\|_{\sup}\leq 1$ such that $\mathcal{E}_{K_i}(v-u)^{1/2}<4\delta$.

Let $(v_n)_{n=1}^\infty\subset\mathcal{S}_{\partial\Omega}$ with $v_n=F_i^{(n)}\circ h_i$ be such that $\lim_{n\to\infty}\mathcal{E}_{\partial\Omega}(v_n-u)=0$. Then by a similar reasoning as in the proof of Proposition \ref{P:closedop} we observe $\lim_{n\to\infty} ((F_i^{(n)})'\circ h_i)^2=\Gamma(u)$ in $L^1(K_i,\mu)$. Therefore, if for sufficiently large $n$ we set $\widetilde{F}_i:=F_i^{(n)}$, we can ensure that $\mathcal{E}_{K_i}(\widetilde{F}_i\circ h_i-u)<\delta^2$ and, by uniform integrability, that 
\[\sup_n\int_{A_\delta}(F_i^{(n)})'(h_i(x))^2\mu(dx)<\frac{\delta^2}{2},\]
where $A_\delta=\left\lbrace x\in K_i: |\widetilde{F}_i'(h_i(x))|\geq 1+\delta\right\rbrace$.
Note that for sufficiently large $n$ the measure $\mu(A_\delta)$ of $A_\delta$ will be smaller than any given $\delta'>0$, because
\[\mu(A_\delta)\leq \mu\left(\left\lbrace x\in K_i:|\widetilde{F}_i'(h_i(x))^2-\Gamma(u)(x)|\geq \delta^2\right\rbrace\right)\leq \frac{1}{\delta^2}\int_{K_i}|\widetilde{F}_i'(h_i(x))^2-\Gamma(u)(x)|\mu(dx).\]
Next, let $\varphi\in C^1(\mathbb{R})$ be a function with $\left\|\varphi\right\|_{\sup}\leq 1$ such that $|\varphi(s)|\leq |(s\wedge 1)\vee(-1)|$ for all $s$ and $\big\|\varphi(\widetilde{F}_i')-(\widetilde{F}_i'\wedge 1)\vee (-1)\big\|_{\sup}<\delta$. Then for the function
\[F_i(s):=\widetilde{F}_i(0)+\int_0^s\varphi(\widetilde{F}_i')(t)dt\]
we have $\sup_{x\in K_i\setminus A_\delta}|F_i'(h(x))-\widetilde{F}_i'(h(x))|<2\delta$ and
\[\int_{A_\delta} F_i'(h_i(x))^2\mu(dx)\leq \int_{A_\delta}\widetilde{F}_i'(h_i(x))^2\mu(dx)<\frac{\delta^2}{2}.\]
Combining, we observe 
\begin{align}
\mathcal{E}_{K_i}(F_i\circ h_i-u)^{1/2}&\leq \mathcal{E}_{K_i}(\widetilde{F}_i\circ h_i-u)^{1/2}+\left(\int_{K_i\setminus A_\delta}|\widetilde{F}_i'(h_i(x))-F_i'(h_i(x))|^2\mu(dx)\right)^{1/2}\notag\\
&\ \ \ \ \ \ \ \ \ +\left(4\int_{A_\delta} (\widetilde{F}_i'(h_i(x))^2 + F_i'(h_i(x))^2)\mu(dx)\right)^{1/2}\notag\\
&\leq 4\delta.\notag
\end{align}
Now let $v$ be the function on $\partial\Omega$ that equals $F_i\circ h_i$ on $K_i$ and is linear on $K_i^c$. Clearly $v\in\mathcal{D}(\mathcal{E}_{\partial\Omega})$, and since $\sup_{x,y\in K_i}|v(x)-v(y)|\leq 1$, we also have $\Gamma(v)\leq 1$ $\mu$-a.e. on $\partial\Omega\setminus K_i$.
\end{proof}

We call a function \emph{$\mathcal{E}_{\partial\Omega}$-intrinsically Lipschitz} if it is Lipschitz with respect to $d_{\mathcal{E}_{\partial\Omega}}$. Consequently a function $u\in\mathcal{D}(\mathcal{E}_{\partial\Omega})$ is $\mathcal{E}_{\partial\Omega}$-intrinsically Lipschitz if and only if there is a constant $L>0$ such that $|u(x)-u(y)|\leq L|h_i(x)-h_i(y)|$ for any $x,y\in \mathring{K}_i$ and any $i$. The following is the main result of this section.
   
   \begin{theorem}\label{thm:lip}
      Any \EpO-intrinsically Lipschitz function on \pO\ has an Euclidean-Lipschitz extension to \oO.
         \end{theorem} 
         
         This theorem has implications for the Dirichlet form $\E_0$ as in (\ref{E:unperturbed}). To functions from $\mathcal{S}_{\partial\Omega}$ we refer as \emph{$\mathcal{E}_{\partial\Omega}$-intrinsic $C^1$-functions}.
         
\begin{corollary}\label{cor:coord} \mbox{} 
\begin{enumerate}
\item[(i)]  The Dirichlet form $\mathcal{E}_0$ has a core consisting of functions in $C(\oO)$ whose restriction to $\Omega$ is in $C^1(\Omega)$, and whose restriction to $\partial\Omega$ is a $\mathcal{E}_{\partial\Omega}$-intrinsic $C^1$-function. For a core function $u$ the gradients $\nabla u$ and $D_{\partial\Omega}u$ are well defined pointwise in $\Omega$ and on $\partial\Omega$, respectively.
\item[(ii)] There are two coordinate functions $y_1,y_2$ 
   which are contained in the core and separate the points of  \oO. 
   \end{enumerate}
\end{corollary}
   
\begin{remark}\mbox{}
\begin{enumerate}
\item[(i)] In particular, for a core function $u$ we have that $D_{\partial\Omega} u$ is continuous on $\partial \Omega$ and $\nabla u$ is continuous on $\Omega$, and the derivatives can be computed by the usual calculus rules. We do not claim that $\nabla u$ is continuous on \oO. Note that a typical function in Euclidean $C^1(\oO)$ is not in the domain of \E.
\item[(ii)]   The coordinate functions in Corollary \ref{cor:coord} can be used to 
   obtain a coordinate representation for the gradients. 
   \end{enumerate}
\end{remark}

Before proving Theorem \ref{thm:lip}, we introduce general notation we will be using. For notational convenience we assume that the domain $\Omega$ pictured in Figure~\ref{fig-snow} is embedded in the complex plane 
$\mathbb C$ and inscribed in the  circle of radius $\sqrt3$. 
In particular, the six most outward  points of \oO\ are   $\sqrt3e^{ik\pi/3+i\pi/6}$, $k=0,1,...,5$. 
The inscribed circle is the unit circle. 
In particular, the six most inward  points of \pO\ are the 6th roots of unity  $e^{ik\pi/3}$, $k=0,1,...,5$. 
The most left and right most inward points of \pO\ are $-1,1\in\mathbb C$ respectively.

Although we now work with Koch curves that, in comparison with those in Section \ref{S:tangential}, are dilated by the factor $3$, we will continue to use the notation of Section \ref{S:tangential}. For example, two $p$ and $q$ from $V_{n+1}(\partial\Omega)$ that are neighbors, $p\sim_{n+1} q$, now have Euclidean distance $3^{-n}$. Let us write $V_n(\partial\Omega):=\bigcup_{i=1}^6 V_n(K_i)$, where $V_n(K_i)$ is defined for $K_i$ in the same way as $V_n(K)$ for $K$.

\subsection{A first proof of Theorem~\ref{thm:lip}}\label{subsub:lip}

Let $f$ be an \EpO-intrinsically Lipschitz function on \pO. We would like to construct 
a Euclidean-Lipschitz extension of $f$ to \oO, which we denote $g$. There are several natural constructions based on the natural ``approximate'' triangulations of $\Omega$.  
One of them is reminiscent of the study of the structure of $\varepsilon$-neighborhoods of the Koch curve described in \cite{LP06}, and also of the constructions in \cite{E1,E2}. It uses a weakly self-similar triangulation, \emph{which is nicely separated from the boundary} of the fractal 
 snowflake domain 
  $ %\overline
  {\Omega}$.

For this first construction, we consider the lattice $L_n=3^{-n}\mathbb Z \{e^{ik\pi/3}|k=0,1,...,5\}$. 
It is easy to see that for each $x\in L_n\cap\Omega$ there are finitely many points of $L_n\cap\pO$ that are closest to $x$ (this number of the closest boundary points ranges from six for the center point $x=0\in\Omega$, to one for most of points of $x\in L_n\cap\Omega$). On $L_n\cap\Omega$ we obtain a function $g_n$ by defining $g_n(x)$ to be the average of $f$ at these points of $L_n\cap\pO$ closest to $x\in L_n\cap\Omega$.
Note that for any $n$ we have $L_n\cap\pO=V_{n+1}(\partial\Omega)$ (in the notation of the present section).

Naturally, we assume that the points of each lattice $L_n$ are the vertices of closed equilateral triangles $T_{n,m}$ of sides $3^{-n}$, and so each lattice $L_n$
defines a triangulation of $\mathbb R^2=\cup_mT_{n,m}$ into such triangles $T_{n,m}$.
We define $\T_n:=\bigcup_{m: T_{n,m}\subset\Omega} T_{n,m}$ as the union of triangles $T_{n,m}$ that lie inside $\Omega$. Note that $\T_n$  is a compact set contained in  $\Omega$, and so $\T_n$  is separated from \pO. See Figures \ref{F:T1} and \ref{F:T2}.
As usual, we denote the boundary of $\T_n$ by $\partial\T_n$.

  \begin{figure}
  \scalebox{.5}{
 \def\HEX{ \put(0,0){\line(1, 0){12}}
 	\put(0,0){\line(3, 5){6}}
 	\put(12,0){\line(-3, 5){6}}
 	\put(0,0){\line(-1, 0){12}}
 	\put(0,0){\line(-3, 5){6}}
 	\put(-12,0){\line(3, 5){6}}
 	\put(-6,10){\line(1, 0){12}}
 	\put(0,0){\line(1, 0){12}}
 	\put(0,0){\line(3, -5){6}}
 	\put(12,0){\line(-3, -5){6}}
 	\put(0,0){\line(-1, 0){12}}
 	\put(0,0){\line(-3, -5){6}}
 	\put(-12,0){\line(3, -5){6}}
 	\put(-6,-10){\line(1, 0){12}}}
 \def\TT{\multiput(0,0)(6,10){2}{
 		\multiput(0,0)(12,0){2}{
 			\put(0,0){\line(3, 5){6}}
 	\put(12,0){\line(-3, 5){6}}
 	\put(0,0){\line(1, 0){12}}
 	\put(12,0){\line(3, 5){6}}
 	\put(6,10){\line(1, 0){12}}
 	}}}
 \def\TTu{ 	\put(0,0){\line(1, 0){24}}
 			\put(0,0){\line(3, 5){12}}
 			\put(24,0){\line(-3, 5){12}}
 			\put(12,0){\line(3, 5){6}}
 			\put(12,0){\line(-3, 5){6}}
 			\put(6,10){\line(1, 0){12}}
 		 		}
 \def\TTd{ 	\put(0,0){\line(1, 0){24}}
 	\put(0,0){\line(3, -5){12}}
 	\put(24,0){\line(-3, -5){12}}
 	\put(12,0){\line(3,- 5){6}}
 	\put(12,0){\line(-3, -5){6}}
 	\put(6,-10){\line(1, 0){12}}
 }
  \begin{picture}(324, 360)(-162, -180)
 \put(-175,-190)
 {\includegraphics[width=350pt,height=380pt]{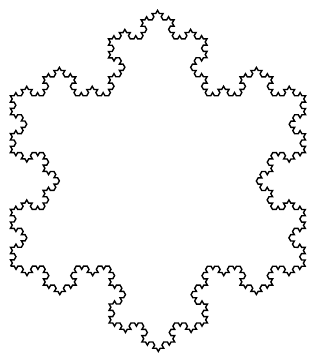}}
 \setlength{\unitlength}{3pt}
\put(0,0){\TTu\TTd}
\put(-24,0){\TTu\TTd}
\put(-12,20){\TTu\TTd}
\put(-12,-20){\TTu\TTd}
\put(12,-20){\TTu}
\put(-36,-20){\TTu}
\put(12,20){\TTd}
\put(-36,20){\TTd}
  \end{picture}
  }
 \caption{The set $\T_1\subset \Omega$, triangulated by triangles of side length $\frac13$.}
  \label{F:T1}
 \end{figure}

From the functions $g_n$ on $L_n\cap\Omega$ we now inductively define $g$ as follows. On $\T_1$ we define $g$ as piece-wise linear interpolation of $g_1$ on the triangulation by triangles $T_{1,m}\subset\T_1$. Next, we extend $g$ from  $\T_1$ to $\T_2$ 
using piece-wise linear interpolation of $g_2$ on the triangles  $T_{2,m}$ contained in $\T_2$ but not in $\T_1$. Inductively, if $g$ is already defined on $\T_{n-1}$, we extend it to  $\T_n$ 
using piece-wise linear interpolation of $g_n$ on the triangles $T_{n,m}$ contained in $\T_{n}$ but not in $\T_{n-1}$. 

In other words, on the ``shell'' $\partial\T_n$ we define $g$ as the unique piecewise linear interpolation of $g_n$ restricted to $\partial\T_n$. In the region between 
$\partial\T_{n-1}$ and $\partial\T_n$ we define 
$g$ by interpolating linearly on the triangles $T_{n,m}$ 
that lie in $\T_n$ but not in $\T_{n-1}$. 

The resulting function $g$ is an Euclidean-Lipschitz extension of $f$ 
because of scaling. When we pass from $n$ to $n+1$, the \EpO-intrinsic  Lipschitz 
constant of function $f$ is scaled by the factor $\frac14$, but the 
 Euclidean-Lipschitz constant of function $g$ is scaled by 
 the factor $\frac13$. More precisely, for any $p\in V_n(\partial\Omega)$ we have
\begin{equation}\label{E:lipschitzratio}
L_{\mathcal{E}_{\partial\Omega}}(f)\geq \sup_{q\in V_n(\partial\Omega), q\sim_n p}\frac{|f(p)-f(q)|}{d_{\mathcal{E}_{\partial\Omega}}(p,q)}=\left(\frac43\right)^n \sup_{q\in V_n(\partial\Omega), q\sim_n p}\frac{|f(p)-f(q)|}{|p-q|},
\end{equation}
where $L_{\mathcal{E}_{\partial\Omega}}(f)$ denotes the $\mathcal{E}_{\partial\Omega}$-intrinsic Lipschitz constant of $f$. This follows from Corollary \ref{C:identifymetric} together with formula (\ref{E:Moscoformula}). For each $x\in L_n\cap \partial \T_n$ the value $g(x)=g_n(x)$ is obtained as an average of the values of $f$ on the points from $V_{n+1}(\partial\Omega)$ closest to $x$, and these points all have Euclidean distance $3^{-n}$ from $x$.

If $x,x'\in \partial \T_n$ have Euclidean distance $3^{-n}$ then there is a chain of maximally $8$ consecutive points from $V_{n+1}(\partial\Omega)$ that contains all points recruited to compute the averages $g_n(x)$ and $g_n(x')$. Figure \ref{F:Av} displays this most extreme case. Using (\ref{E:lipschitzratio}) and the triangle inequality we can then see that the maximal difference of the values on $f$ on these points is bounded by $ 4^{-n} 7\:L_{\mathcal{E}_{\partial\Omega}}(f)$. By averaging then also the difference of $g_n(x)$ and $g_n(x')$ is bounded by that number. Consequently the Euclidean-Lipschitz constant of the function $g$ on $\partial\T_n$ is controlled by $\left(\frac34\right)^n\:7\:L_{\mathcal{E}_{\partial\Omega}}(f)$. If $x\in \partial L_n\cap \T_n$ and $x'\in L_{n-1}\cap \partial\T_{n-1}$ have Euclidean distance $2\cdot 3^{-n}$ then the maximal Euclidean distance of the points on $V_{n+1}(\partial\Omega)$ recruited to compute the averages $g_n(x)$ and $g_{n-1}(x')$ is certainly bounded by $3^{-n+2}$. Proceeding similarly to see that there is a constant $c>0$, not depending on $x$, $x'$, $n$ or $f$ such that the Euclidean-Lipschitz norm of the function $g$ on $\T_n\setminus \mathring{\T}_{n-1}$ is controlled by $c\left(\frac34\right)^n\:L_{\mathcal{E}_{\partial\Omega}}(f)$.
 
\begin{remark} In some sense Corollary \ref{C:identifymetric} and formula (\ref{E:lipschitzratio}) imply not only that 
 $g$ is Euclidean-Lipschitz, but also that $g$ has ``zero (Euclidean) tangential derivative'' along $\partial\Omega$, because an   
 \EpO-intrinsic  Lipschitz function has to grow ``infinitely slow'' in the ``tangential direction'' to the Koch curve. 
\end{remark}

  \begin{figure}
  \scalebox{.5}{
 \def\HEX{ \put(0,0){\line(1, 0){12}}
 	\put(0,0){\line(3, 5){6}}
 	\put(12,0){\line(-3, 5){6}}
 	\put(0,0){\line(-1, 0){12}}
 	\put(0,0){\line(-3, 5){6}}
 	\put(-12,0){\line(3, 5){6}}
 	\put(-6,10){\line(1, 0){12}}
 	\put(0,0){\line(1, 0){12}}
 	\put(0,0){\line(3, -5){6}}
 	\put(12,0){\line(-3, -5){6}}
 	\put(0,0){\line(-1, 0){12}}
 	\put(0,0){\line(-3, -5){6}}
 	\put(-12,0){\line(3, -5){6}}
 	\put(-6,-10){\line(1, 0){12}}}
 \def\TT{\multiput(0,0)(6,10){2}{
 		\multiput(0,0)(12,0){2}{
 			\put(0,0){\line(3, 5){6}}
 	\put(12,0){\line(-3, 5){6}}
 	\put(0,0){\line(1, 0){12}}
 	\put(12,0){\line(3, 5){6}}
 	\put(6,10){\line(1, 0){12}}
 	}}}
 \def\TTu{ 	\put(0,0){\line(1, 0){24}}
 			\put(0,0){\line(3, 5){12}}
 			\put(24,0){\line(-3, 5){12}}
 			\put(12,0){\line(3, 5){6}}
 			\put(12,0){\line(-3, 5){6}}
 			\put(6,10){\line(1, 0){12}}
 		 		}
 \def\TTd{ 	\put(0,0){\line(1, 0){24}}
 	\put(0,0){\line(3, -5){12}}
 	\put(24,0){\line(-3, -5){12}}
 	\put(12,0){\line(3,- 5){6}}
 	\put(12,0){\line(-3, -5){6}}
 	\put(6,-10){\line(1, 0){12}}
 }
  \begin{picture}(324, 360)(-162, -180)
 \put(-175,-190)
 {\includegraphics[width=350pt,height=380pt]{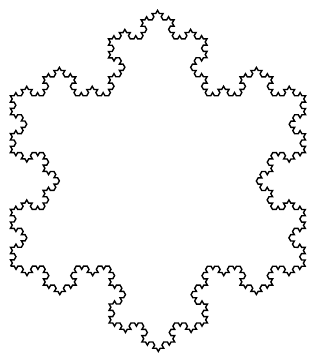}}
 \setlength{\unitlength}{3pt}
\put(0,0){\TTu\TTd}
\put(-24,0){\TTu\TTd}
\put(-12,20){\TTu\TTd}
\put(-12,-20){\TTu\TTd}
\put(12,-20){\TTu}
\put(-36,-20){\TTu}
\put(12,20){\TTd}
\put(-36,20){\TTd}
 \setlength{\unitlength}{1pt}
\put(36,60){\multiput(0,0)(24,0){4}{\TTu}}
\put(24,80){\multiput(0,0)(24,0){5}{\TTd}}
\put(36,60){\multiput(0,0)(-12,20){5}{\TTu}}
\put(24,80){\multiput(0,0)(-12,20){4}{\TTd}}
\put(-60,60){\multiput(0,0)(-24,0){4}{\TTu}}
\put(-48,80){\multiput(0,0)(-24,0){5}{\TTd}}
\put(-60,60){\multiput(0,0)(12,20){5}{\TTu}}
\put(-48,80){\multiput(0,0)(12,20){4}{\TTd}}
\put(24,-80){\multiput(0,0)(24,0){5}{\TTu}}
\put(36,-60){\multiput(0,0)(24,0){4}{\TTd}}
\put(24,-80){\multiput(0,0)(-12,-20){4}{\TTu}}
\put(24,-80){\multiput(0,0)(-12,-20){4}{\TTd}}
\put(-48,-80){\multiput(0,0)(-24,0){5}{\TTu}}
\put(-60,-60){\multiput(0,0)(-24,0){4}{\TTd}}
\put(-48,-80){\multiput(0,0)(12,-20){4}{\TTu}}
\put(-48,-80){\multiput(0,0)(12,-20){4}{\TTd}}
\put(72,0){\multiput(0,0)(12,20){4}{\TTu\TTd}}
\put(72,0){\multiput(0,0)(12,-20){4}{\TTu\TTd}}
\put(-96,0){\multiput(0,0)(-12,-20){4}{\TTu\TTd}}
\put(-96,0){\multiput(0,0)(-12,20){4}{\TTu\TTd}}
\put(96,80){\TTu}
\put(-120,80){\TTu}
\put(120,40){\TTu}
\put(-144,40){\TTu}
\put(12,-140){\TTu}
\put(-36,-140){\TTu}
\put(96,-80){\TTd}
\put(-120,-80){\TTd}
\put(120,-40){\TTd}
\put(-144,-40){\TTd}
\put(12,140){\TTd}
\put(-36,140){\TTd}
  \end{picture}
  }
 \caption{The set $\T_2\subset \Omega$, with $\T_1\subset \T_2$ triangulated by triangles of side length $\frac13$ and $\T_2\setminus \mathring{\T_1}$ triangulated by triangles of side length $\frac19$. }
  \label{F:T2}
 \end{figure} 

 \begin{figure}
  \scalebox{.5}{
\begin{tikzpicture}[decoration=Koch snowflake]
    \draw[line width = 3pt] decorate{ decorate{ decorate{ decorate{ (0,0) -- (27,0) }}}};
\draw[black, line width =2pt] (12,0) -- (15,0);
\draw[black, line width =2pt] ($(13.5,{sqrt(6.75)})$) -- ($(16.5,{-sqrt(6.75)})$) ;
\draw[black, line width =2pt] ($(13.5,{sqrt(6.75)})$) -- ($(10.5,{-sqrt(6.75)})$) ;
\draw[black, line width =2pt] ($(13.5,{-sqrt(6.75)})$) -- (15,0);
\draw[black, line width =2pt] ($(13.5,{-sqrt(6.75)})$) -- (12,0);
\draw[black, line width =2pt] ($(4.5,{-sqrt(6.75)})$) -- ($(22.5,{-sqrt(6.75)})$);

\draw[black, dashed] (9,0) -- (12,0);
\draw[black, dashed] ($(10.5,{sqrt(6.75)})$) -- (12,0);
\draw[black, dashed]  ($(13.5,{sqrt(6.75)})$) -- ($(10.5,{sqrt(6.75)})$);
\draw[black, dashed]  ($(13.5,{sqrt(6.75)})$) -- ($(16.5,{sqrt(6.75)})$);
\draw[black, dashed]  ($(13.5,{sqrt(6.75)})$) -- ($(12,{2*sqrt(6.75)})$);
\draw[black, dashed]  ($(13.5,{sqrt(6.75)})$) -- ($(15,{2*sqrt(6.75)})$);

\filldraw 
(12,0) circle (5pt) node[align=left, below] {.};
\node at (12,-0.8) {{\Large \bf x} };
\filldraw 
(9,0) circle (5pt) node[align=left, below] {.};
\filldraw 
($(10.5,{sqrt(6.75)})$) circle (5pt) node[align=left, below] {.};
\filldraw 
($(13.5,{sqrt(6.75)})$) circle (5pt) node[align=left, right] {.};
\node at ($(13.5,{sqrt(6.75)-0.8})$) {{\Large \bf x'}};
\filldraw 
($(16.5,{sqrt(6.75)})$) circle (5pt) node[align=left, right] {.};
\filldraw 
($(12,{2*sqrt(6.75)})$) circle (5pt) node[align=left, right] {.};
\filldraw 
($(15,{2*sqrt(6.75)})$) circle (5pt) node[align=left, right] {.};

\filldraw 
($(13.5,{3*sqrt(6.75)})$) circle (5pt) node[align=left, right] {.};
\filldraw 
($(9,{2*sqrt(6.75)})$) circle (5pt) node[align=left, right] {.};
\filldraw 
($(18,{2*sqrt(6.75)})$) circle (5pt) node[align=left, right] {.};

\end{tikzpicture}
}
 \caption{Chain of $8$ consecutive points from $V_{n+1}(\partial\Omega)$ that contains all points considered when computing the values $g_n(x)$ and $g_n(x')$ for neighbours $x$ and $x'$ on the shell $\partial\T_n$. }
  \label{F:Av}
\end{figure}
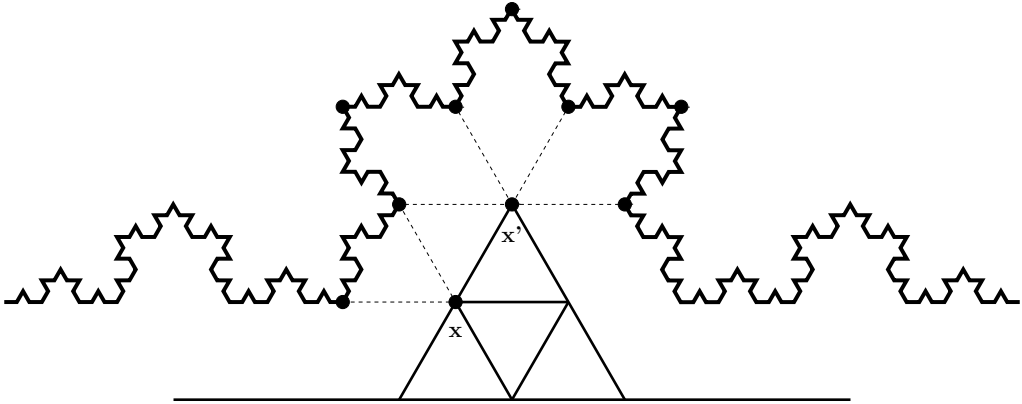

\subsection{A second proof of Theorem~\ref{thm:lip}}\label{subsub:lip2}
We also can present another, different, `pointwise' construction of $g$
 along the following steps. It uses a weakly self-similar triangulation, \emph{which is not nicely separated from the boundary} $\partial\Omega$ of the fractal snowflake domain 
$ %\overline
  		{\Omega}$.

\emph{Step 1}: 
We define $g(0)$ to be the average of $f$ 
at the  6th roots  of unity  $e^{ik\pi/3}$, $k=0,1,...,5$.  
  At each  6th root  of unity  $e^{ik\pi/3}$, $k=0,1,...,5$, we define 
$g(e^{ik\pi/3})=f(e^{ik\pi/3})$. 
After that we interpolate $g$ linearly in each unit equilateral triangle 
with vertices $0, e^{ik\pi/3}, e^{i(k+1)\pi/3}$. 
This defines $g$ in the regular closed convex unit hexagon 
which is the convex span of the  
6th root  of unity  $e^{ik\pi/3}$, $k=0,1,...,5$. 
We denote this closed hexagon $\overline{\Omega}_1$.

\begin{figure}
\begin{center}
\scalebox{.5}{
  	\def\HEX{ \put(0,0){\line(1, 0){12}}
  		\put(0,0){\line(3, 5){6}}
  		\put(12,0){\line(-3, 5){6}}
  		\put(0,0){\line(-1, 0){12}}
  		\put(0,0){\line(-3, 5){6}}
  		\put(-12,0){\line(3, 5){6}}
  		\put(-6,10){\line(1, 0){12}}
  		\put(0,0){\line(1, 0){12}}
  		\put(0,0){\line(3, -5){6}}
  		\put(12,0){\line(-3, -5){6}}
  		\put(0,0){\line(-1, 0){12}}
  		\put(0,0){\line(-3, -5){6}}
  		\put(-12,0){\line(3, -5){6}}
  		\put(-6,-10){\line(1, 0){12}}}
  	\begin{picture}(324, 360)(-162, -180)
  	\put(-175,-190)
  	{\includegraphics[width=350pt,height=380pt]{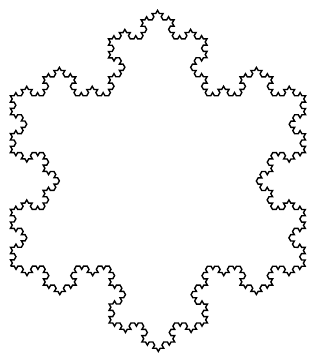}}
  	\setlength{\unitlength}{9pt}
  	\HEX
  	\end{picture}
  	}
  	\caption{The triangulated hexagon $\overline{\Omega}_1$.}
  	\label{F:Omega_1}  	
  	\end{center}
  \end{figure} 

\emph{Step 2}: Note that $\Omega\setminus\overline{\Omega}_1$ 
consists of six disjoint isometric open sets. 
The closure of each of these open sets contains a regular closed convex hexagon 
with sides $\frac13$. On the facet of this hexagon that is part of the boundary of $\overline{\Omega}_1$ the function $g$ is already defined, and naturally we do not change these values. In particular, $g$ is already defined on the vertices that are incident to this facet. The other four vertices of this hexagon lie 
on \pO, and so we define $g(x)=f(x)$ at these vertices. 
Now as we defined $g$ at the all vertices of the hexagon, 
we define $g$ at its center as the average at the vertices. Since the hexagon is naturally triangulated into six equilateral triangles of side $\frac13$, we can interpolate $g$ linearly in each of these triangles. 
Now we have defined $g$ in the closed set which is the union of the unit regular hexagon $\overline{\Omega}_1$
and the six adjacent closed hexagons with sides $\frac13$, We denote this union by $\overline{\Omega}_2$. 

\begin{figure}
\begin{center}
\scalebox{.5}{
  	\def\HEX{ \put(0,0){\line(1, 0){12}}
  		\put(0,0){\line(3, 5){6}}
  		\put(12,0){\line(-3, 5){6}}
  		\put(0,0){\line(-1, 0){12}}
  		\put(0,0){\line(-3, 5){6}}
  		\put(-12,0){\line(3, 5){6}}
  		\put(-6,10){\line(1, 0){12}}
  		\put(0,0){\line(1, 0){12}}
  		\put(0,0){\line(3, -5){6}}
  		\put(12,0){\line(-3, -5){6}}
  		\put(0,0){\line(-1, 0){12}}
  		\put(0,0){\line(-3, -5){6}}
  		\put(-12,0){\line(3, -5){6}}
  		\put(-6,-10){\line(1, 0){12}}}
  	\begin{picture}(324, 360)(-162, -180)
  	\put(-175,-190)
  	{\includegraphics[width=350pt,height=380pt]{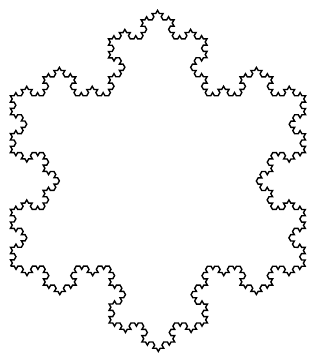}}
  	\setlength{\unitlength}{9pt}
  	\HEX
  	\setlength{\unitlength}{3pt}
  	\put(0,40){\HEX}
  	\put(0,-40){\HEX}
  	\put(36,20){\HEX}
  	\put(36,-20){\HEX}
  	\put(-36,20){\HEX}
  	\put(-36,-20){\HEX}
  	\end{picture}
  	}
  	\caption{The triangulated set $\overline{\Omega}_2$.}
  	\label{F:Omega_2}  	
  	\end{center}
  \end{figure} 

\emph{Step 3}: Note that $\Omega\setminus\overline{\Omega_2}$ 
consists of 30 open components. 
These components come in two shapes: 
18 components are $\frac{1}{3}$ in size in comparison to 
   the components considered in Step 2; 
the other 12 component are of the same scale, but have only one fractal side, rather than two fractal sides. 
In any case, the closure of each of these 30 
components has an inscribed 
  regular closed convex hexagon 
with sides $\frac19$. On each of these hexagons we can define $g$ similarly as in Step 2. Let $\overline{\Omega}_3$ be the union of all these closed hexagons and $\overline{\Omega}_2$. See Figure \ref{F:cont_ind}.

It is easy to proceed by induction afterwards. If for each $n$ we denote the interior of $\overline{\Omega}_n$ by $\Omega_n$, then the increasing domains $\Omega_n$ exhaust $\Omega$ from within.

This theorem is clearly local, and similarly as in the first proof,   
the required properties of the extension function $g$ are again implied by scaling.

\begin{figure}
\begin{center}
\scalebox{.5}{
  	\def\HEX{ \put(0,0){\line(1, 0){12}}
  		\put(0,0){\line(3, 5){6}}
  		\put(12,0){\line(-3, 5){6}}
  		\put(0,0){\line(-1, 0){12}}
  		\put(0,0){\line(-3, 5){6}}
  		\put(-12,0){\line(3, 5){6}}
  		\put(-6,10){\line(1, 0){12}}
  		\put(0,0){\line(1, 0){12}}
  		\put(0,0){\line(3, -5){6}}
  		\put(12,0){\line(-3, -5){6}}
  		\put(0,0){\line(-1, 0){12}}
  		\put(0,0){\line(-3, -5){6}}
  		\put(-12,0){\line(3, -5){6}}
  		\put(-6,-10){\line(1, 0){12}}}
  	\begin{picture}(324, 360)(-162, -180)
  	\put(-175,-190)
  	{\includegraphics[width=350pt,height=380pt]{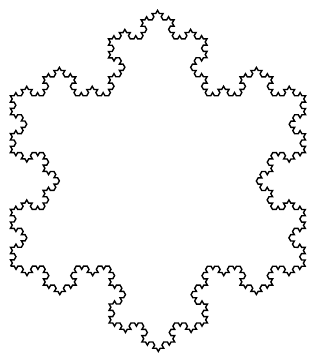}}
  	\setlength{\unitlength}{9pt}
  	\HEX
  	\setlength{\unitlength}{3pt}
  	\put(0,40){\HEX}
  	\put(0,-40){\HEX}
  	\put(36,20){\HEX}
  	\put(36,-20){\HEX}
  	\put(-36,20){\HEX}
  	\put(-36,-20){\HEX}
  	\setlength{\unitlength}{1pt}
  	\put(0,160){\HEX}
  	\put(36,140){\HEX}
  	\put(36,100){\HEX}
  	\put(72,80){\HEX}
  	\put(108,100){\HEX}
  	\put(144,80){\HEX}
  	\put(144,40){\HEX}
  	\put(108,20){\HEX}
  	\put(0,-160){\HEX}
  	\put(36,-140){\HEX}
  	\put(36,-100){\HEX}
  	\put(72,-80){\HEX}
  	\put(108,-100){\HEX}
  	\put(144,-80){\HEX}
  	\put(144,-40){\HEX}
  	\put(108,-20){\HEX}
  	\put(-0,160){\HEX}
  	\put(-36,140){\HEX}
  	\put(-36,100){\HEX}
  	\put(-72,80){\HEX}
  	\put(-108,100){\HEX}
  	\put(-144,80){\HEX}
  	\put(-144,40){\HEX}
  	\put(-108,20){\HEX}
  	\put(-0,-160){\HEX}
  	\put(-36,-140){\HEX}
  	\put(-36,-100){\HEX}
  	\put(-72,-80){\HEX}
  	\put(-108,-100){\HEX}
  	\put(-144,-80){\HEX}
  	\put(-144,-40){\HEX}
  	\put(-108,-20){\HEX}
  	\end{picture}
  	}
  	\caption{The triangulated set $\overline{\Omega}_3$.}
  	\label{F:cont_ind}  	
  	\end{center}
  \end{figure}

\subsection{Proof of Corollary~\ref{cor:coord}}
The first statement (i) 
of this corollary can be  proved by smoothing 
the constructed Euclidean-Lipschitz extensions of $\mathcal{E}_{\partial\Omega}$-intrinsic $C^1$-functions. 

The second statement (ii) can be proved by first constructing Lipschitz coordinates which in a sense may be seen as ``deformed polar coordinates'', and then smoothing them.
We explain the construction of the Lipschitz coordinates only briefly  because the proof of Theorem~\ref{thm:lip}. 

First, we note that it is enough to construct coordinates only locally. 
To construct $y_1$ we can begin with a function which is linear in the intrinsic distance on an open subset of $\partial\Omega$. For instance, we can consider the open subset $\mathring{K}_i$ and a function of form $F\circ h_i$ with $F:\mathbb{R}\to\mathbb{R}$ being an linear function. We can then consider its Euclidean-Lipschitz extension to a sectorial part of $\Omega$ similarly as
constructed in the first proof of Theorem~\ref{thm:lip}. It is easy to see that the extension is locally skew-symmetric (up to a local constant) with respect to the natural local symmetries (if we from some point $x\in \mathring{K}_i$ we go the same distance to the left and to the right, the differences of the function values seen there and $F\circ h_i(x)$ will be equal in modulus but have different sign). In a sense, the function $y_1$ may be seen as the (local) ``angular coordinate'' of the deformed polar coordinates.

Now it is easy to find a complimentary Lipschitz function $y_2$ which vanishes on \oO\  and separates points on each level set of $y_1$. To describe a possible construction for $y_2$, we use the same notation as in the first proof of Theorem~\ref{thm:lip}.
For each $p\in\partial\T_n$ we define $y_2(p)=3^{-n}$. Between the shells  
$\partial\T_n$ and $\partial\T_{n+1}$ we define $y_2$ by linear interpolation. This may be regarded as the ``radial component'' of our coordinates. Together $y_1$ and $y_2$ separate points. Note that both, the definition of $y_1$ and the definition of $y_2$ display some sort of ``weak self-similarity''. After having constructed these Lipschitz coordinates, one can obtain the $C^1$-coordinates by smoothing. 

The construction is essentially elementary because the weak self-similarity allows to define coordinates in a finitely many points, and then define extensions using translations and scaling. It is especially simple to use the 
  second proof of Theorem~\ref{thm:lip}, Subsection~\ref{subsub:lip}, to construct the piecewise linear version of $y_2$ and then smooth it.

\begin{appendix}

\section{Trace and extension results}\label{App:A}

For the convenience of the reader we provide some more detailed references for known trace and extension results used to see that (\ref{E:Hilbert}) turns $V(\Omega,\partial\Omega)$ into a Hilbert space.

A function $f\in L^1(\Omega)$ can be strictly defined at $x\in \overline{\Omega}$ if 
the limit
\[\widetilde{f}(x):=\lim_{\varepsilon\to 0}\frac{1}{\mathcal{L}^2(B(x,\varepsilon)\cap\Omega)}\int_{B(x,\varepsilon)\cap\Omega}f(y)\mathcal{L}^2(dy)\]
exists. Any $f\in H^1(\Omega)$ can be strictly defined at quasi-every $x\in\overline{\Omega}$ (in the sense of the Newtonian capacity on $\mathbb{R}^2$). Writing $\Tr f:=f|_{\partial\Omega}$, where $f|_{\partial\Omega}=\widetilde{f}(x)$, $x\in\partial\Omega$, the trace operator $f\mapsto \Tr f$ defines a bounded linear operator from $H^1(\Omega)$ onto $B_{\delta/2}^{2,2}(\partial\Omega)$, and it has a bounded linear right inverse $f\mapsto \Ext f$. Here $B_\beta^{2,2}(\partial\Omega)$, $0<\beta<1$, denotes the space of all elements $f$ of $L^2(\partial\Omega,\mu)$ such that
\[\left\|f\right\|_{B_\beta^{2,2}(\partial\Omega)}=\left\|f\right\|_{L^2(\partial\Omega,\mu)}+\left(\int\int_{|x-y|<1}\frac{|f(x)-f(y)|^2}{|x-y|^{\delta+2\beta}}\mu(dx)\mu(dy)\right)^{1/2}<+\infty.\]
See \cite[Theorem 2]{W91} or \cite[Theorem 1, Chapter VII]{JW84} (and also \cite{Tri97} for related methods). These trace and extension results make use of the fact that $\mu$ is a $\delta$-measure, i.e. there is a constant $c>0$ such that $cr^\delta\leq \mu(B(x,r))\leq c^{-1}r^\delta$ for all $x\in\partial\Omega$ and $0<r<1$. Here $B(x,r)$ denotes the open ball (in Euclidean metric) centered at $x$ with radius $r$. These arguments have also been used in \cite[Proposition 2.10]{LaV14} and in similar form in \cite[Theorem 4.6]{La02}. Moreover, by arguments from \cite{LaVi99} the space $\mathcal{D}(\mathcal{E}_{\partial\Omega})$ is seen to be embedded in $B_{\delta/2}^{2,2}(\partial\Omega)$, see also \cite[Proposition 4.9]{La02}. Combining these results, one observes that with (\ref{E:Hilbert}) the space $V(\Omega,\partial\Omega)$ is Hilbert. 

\end{appendix}

\end{document}